\documentclass[a4paper,12pt,twoside]{amsart}

\usepackage{amsfonts}
\usepackage{amsfonts,amssymb,amscd}
\usepackage{graphicx}
 \usepackage[left=2.3cm,top=3.3cm,right=2.3cm,bottom=2.7cm]{geometry}
\usepackage{mathrsfs}
\usepackage{xfrac}
\let\mathcal\mathscr
\usepackage[colorlinks=true,citecolor=blue, urlcolor=blue, breaklinks, linkcolor=black]{hyperref}

\usepackage[all,ps,cmtip]{xy}

\DeclareRobustCommand{\SkipTocEntry}[5]{}

\def\C{{\bf C}}

\def\llra{\hbox to 10mm{\rightarrowfill}}

\def\lllra{\hbox to 15mm{\rightarrowfill}}

\def\PA{{\widehat A}}
\def\PB{{\widehat B}}

\def\PE{{\widehat E}}

\def\PC{{\widehat C}}

\def\phi{{\varphi}}

\def\wX{{\widetilde X}}

\def\wZ{{\widetilde Z}}

\def\cF{\mathcal{F}}
\def\cL{\mathcal{L}}
\def\cO{\mathcal{O}}
\def\cG{\mathcal{G}}

\def\cHom{\mathcal{Hom}}

\def\cC{\mathcal{C}}
\def\cM{\mathcal{M}}

\def\cQ{\mathcal{Q}}

\let\tilde\widetilde

\DeclareMathOperator{\rank}{rank}

\DeclareMathOperator{\Pic}{Pic}

\DeclareMathOperator{\Supp}{Supp}

\DeclareMathOperator{\vol}{vol}

\newtheorem{lemm}{Lemma}[section]
\newtheorem{theo}[lemm]{Theorem}
\newtheorem{coro}[lemm]{Corollary}
\newtheorem{prop}[lemm]{Proposition}
\newtheorem{claim}[lemm]{Claim}

\theoremstyle{definition}
\newtheorem{defi}[lemm]{Definition}
\newtheorem{rema}[lemm]{Remark}

\newtheorem{exam}[lemm]{Example}

\theoremstyle{remark}
\newtheorem*{remark*}{Remark}
\newtheorem*{note*}{Note}

\begin{document}
\title{Some results on the eventual paracanonical maps}
\author[Z.~Jiang]{Zhi Jiang}
\address{Shanghai center for mathematical sciences\\
22 floor, East Guanghua Tower\\
Fudan University, Shanghai}
\email{zhijiang@fudan.edu.cn}

\keywords{the eventual paracanonical maps, M-regular sheaves, abelian covers}

\subjclass[2010]{14E05, 14F17, 14K99, 14J30}
\begin{abstract}
The eventual paracanonical map was introduced by Barja, Pardini, and Stoppino in order to prove refined  Severi-type inequalities. We study the general structures of the eventual paracanonical maps by generic vanishing theory. In particular, we show that the Albanese morphism of $X$ admits very special structure if $X$ is a surface or a threefold and the eventual paracanonical map of $X$ is not birational.
\end{abstract}
\maketitle

\section{Introduction}
Let $X$ be a smooth projective variety.
Assume that $a: X\rightarrow A$ is a primitive morphism to an abelian variety, which is generically finite onto its image, and $L$ is a line bundle on $X$ such that $H^0(X, L\otimes a^*P)\neq 0$ for $P\in\Pic^0(A)$ general, Barja, Pardini and Stoppino \cite{BPS1} defined the eventual map of $L$ with respect to $a$: $$\varphi_{L,a}: X\rightarrow Z_{L,a}$$ such that $\varphi_{L,a}$ is generically finite and $a$ factors birationally through $\varphi_{L,a}$. 

The map $\varphi_{L,a}$ can be constructed the following way. Take  $d>0$ sufficiently large and divisible and let $m_d: A\rightarrow A$ be the multiplication by $d$ and take $\pi: X^{(d)}\rightarrow X$ the induced \'etale cover by $m_d$. Let $\varphi_d$ be the map induced by the linear system $|\pi^*L\otimes P|$ for $P\in \Pic^0(A)$ general. Then $\varphi_d: X^{(d)}\rightarrow Z^{(d)}$ is generically finite onto its image. Moreover the Galois group $G$ of $\pi$ acts naturally  on $Z^{(d)}$, $\varphi_d$ is $G$-equivariant, and the quotient of $\varphi_d$ by $G$ gives us the eventual map $\varphi_{L,a}$ of $L$. 

We  denote by $a_X: X\rightarrow A_X$ the Albanese morphism of $X$. We shall denote by $\varphi_L$ instead of $\varphi_{L, a_X}$. All morphisms $a: X\rightarrow A$ factor through $a_X$ by the universal property of $a_X$. 
However, $\varphi_{L, a}$ does not factor birationally through  $\varphi_{L}$, as we can see in the following example.
 \begin{exam}
 Let $D$ be a smooth projective curve of genus $g(D)\geq 1$ and let $\rho: C\rightarrow D$ be a ramified double cover from a smooth projective curve $C$ to $D$ with $g(C)>2g(D)$. Let $L$ be a line bundle  of degree $g(D)$ on $D$. Then $\rho^*L$ has degree $2g(D)<g(C)$. Hence $h^0(C, \rho^*L\otimes P)=0$ for $P\in\Pic^0(A_C)$ and we cannot define $\varphi_{\rho^*L, a_C}$. On the other hand, for $P_1\in\Pic^0(A_D)$ general, $h^0(C, \rho^*L\otimes \rho^*P_1)=h^0(D, L\otimes P_1)=1$ and  one can check that   $\varphi_{\rho^*L, a}$ is birationally equivalent to $\rho$. 
 \end{exam}
 
Among all line bundles $L$, we are in particular interested in the canonical divisor $K_X$. If $\chi(X, \omega_X)=h^0(X, K_X\otimes P)>0$ for $P\in\Pic^0(X)$ general, we will call the eventual map of $K_X$ with respect to $a_X$ the eventual paracanonical map of $X$ and we use the following notations:
\begin{eqnarray*}
a_X: X\xrightarrow{\varphi_X} Z\xrightarrow{g} A_X,
\end{eqnarray*}
where after birational modifications, we will assume that $Z$ is smooth and $\varphi$ is generically finite and surjective. 

Barja, Pardini and Stoppino proved in \cite{BPS2} that 
\begin{itemize}
\item[(A)]either $\chi(Z, \omega_Z)=0$ or $\chi(Z, \omega_Z)=\chi(X, \omega_X)>0$; \item[(B)]assume $\dim X=2$, then $\deg \varphi_X\leq 4$ and if moreover $\chi(Z, \omega_Z)=\chi(X, \omega_X)>0$, $\deg\varphi_X\leq 2$. 
\end{itemize}
The bound $\deg\varphi_X\leq 4$ for surfaces is optimal, as Barja, Pardini and Stoppino constructed a surface $X$ whose eventual paracanonical map is a birational $(\mathbb{Z}/2\mathbb{Z})^2$-cover of $A_X$. 

Barja, Pardini and Stoppino then asked \cite[Question 5.6]{BPS2}
 whether or not there exist  surfaces $X$ with $\chi(Z, \omega_Z)=\chi(X, \omega_X)$ and $\deg\varphi_X=2$  and if there exit  constants $C(n)$ such that $\deg\varphi_X\leq C(n)$ for all $X$ of dimension $n\geq 3$. The latter question is particularly interesting because the effective bound $\deg\varphi_X\leq 4$ for surfaces relies heavily on the Miyaoka-Yau inequality that $9\chi(X, \omega_X)\geq \vol(X)$. This inequality fails in higher  dimensions. Indeed, Ein and Lazarsfeld found a threefold $X$, which is of maximal Albanese dimension, of general type, and of  $\chi(X, \omega_X)=0$ (\cite{EL}). 
 
 In this article, we will apply generic vanishing theory (\cite{EL, PP, CJ}) to study the eventual paracanonical maps. The starting point is the following easy lemma.

\begin{lemm}\label{characterization}
The eventual paracanonical map $\varphi_X: X\rightarrow Z$  can be characterized as the map with maximal degree such that  
 \begin{itemize}
\item[(1)] $a_X$ factors birationally through $\varphi_X$;
\item[(2)] there exists a rank $1$ subsheaf $\cL$ of $\varphi_{X*}\omega_X$ such that $h^0(Z, \cL\otimes Q)=h^0(X, \omega_X\otimes Q)$ for $Q\in\Pic^0(X)$ general.
 \end{itemize}
\end{lemm}

\begin{proof}
From the construction of $\varphi_X$, as $Z^{(d)}$ is the image of the map induced by the linear system $|K_{X^{(d)}}\otimes P|$, there exists a line bundle $\cL^{(d)}$ on $Z^{(d)}$ and a fixed divisor $F$ on $X^{(d)}$, such that $|K_{X^{(d)}}\otimes P|=\varphi_d^*|\cL^{(d)}\otimes P|+F$, for $P\in\Pic^0(A_X)$ general. Moreover, we have the commutative diagram
\begin{eqnarray*}
\xymatrix{
X^{(d)}\ar[d]^{\pi_X}\ar[r]^{\varphi_{d}} & Z^{(d)}\ar[d]^{\pi_Z}\ar[r] & A_X\ar[d]^{m_d}\\
X\ar[r]^{\varphi_X} & Z\ar[r] & A_X,
}
\end{eqnarray*}
where after birational modifications of $X$ and $Z$, all vertical morphisms are \'etale, and we know that
the line bundle $\cL^{(d)}$ descends to a line bundle $\cL$ on $Z$. As $\cL\hookrightarrow \varphi_{X^*}\omega_X$ and $h^0(X^{(d)}, K_{X^{(d)}}\otimes P)=h^0(Z^{(d)}, \pi_Z^*\cL\otimes P)$ for $P\in\Pic^0(A_X)$ general, we conclude that $h^0(X, K_X\otimes Q)=h^0(Z, \cL\otimes Q)$ for $Q\in\Pic^0(A_X)$ general. 

On the other hand, for any $Y$ satisfies conditions $(1)$ and $(2)$, we see that the map on $X^{(d)}$ induced by $|K_{X^{(d)}}\otimes P|$ factors through the corresponding \'etale cover $Y^{(d)}$ of $Y$ and hence $\varphi_X$ factors through $Y$. Thus, $\varphi_X$ is the map with maximal degree satisfying  conditions $(1)$ and $(2)$.
\end{proof}

\begin{rema}\label{va1}
By the same argument in the proof of Lemma \ref{characterization}, $\varphi_{L, a}: X\rightarrow Z_{L, a}$ can be characterized  as the map with maximal degree such that 
\begin{itemize}
\item[(1)] $a$ factors birationally  through $\varphi_{L, a}$;
\item[(2)] there exists a line bundle $H$ on $Z_{L,a}$ such that $\varphi_{L, a}^*H\hookrightarrow L$ and $h^0(X, L\otimes P)=h^0(Z_{L, a}, H\otimes P)$ for $P\in\Pic^0(A)$ general.
\end{itemize}
\end{rema}
\begin{rema}
By generic vanishing, we know that if $\chi(X, \omega_X)>0$, then $h^0(X, K_X\otimes P)=h^0(X, K_X\otimes P_1)$ for $P\in\Pic^0(A_X)$ and $P_1\in\Pic^0(A)$ general. Thus, by Lemma \ref{characterization} and the universal property of the Albanese morphism, we conclude that $\varphi_{K_X, a}$ always factors through $\varphi_X$.

Lemma \ref{characterization} can be used to give all possible structures of $\varphi_{K_C, a}$ for a smooth projective curve $C$ together with a primitive moprhism $a: C\rightarrow A$.  We may assume that $Z$ is a smooth projective curve with genus $g(Z)\geq 1$. By the canonical splitting,  $(\varphi_{K_C, a})_{*}\omega_C=\omega_Z\oplus\cF$. Note  that $(\varphi_{K_C, a})_{*}\omega_{C/Z}$ is a nef vector bundle (see for instance \cite[Corollary 3.2]{Vie}).

 If $g(Z)>1$, by Riemann-Roch, for each direct summand $\cQ$ of $(\varphi_{K_C, a})_{*}\omega_C$, we have $h^0(Z, \cQ\otimes P)>0$, for $P\in\Pic^0(A)$ general. Thus, Lemma \ref{characterization} implies that $\cF=0$ and hence $\varphi_{K_C, a}$ is birational.
 
If $g(Z)=1$, we may assume that $Z=A$ and $\varphi_{K_C, a}$ is not birational. Since $a$ is primitive, $\cF$ is an ample vector bundle. Then by Lemma \ref{characterization} and Riemann-Roch, we conclude that $\rank\cF=1$ and hence $a$ is a double cover.
\end{rema}

In higher dimensions, in order to apply Lemma \ref{characterization} to study the sturcute of $\varphi_X$. We need further informations about $a_{X*}\omega_X$. By  the decomposition theorem proved in \cite{CJ} (see also \cite{PPS}),  we know that $$a_{X*}\omega_X=\cF^{X}\bigoplus_{p_B:A_X\twoheadrightarrow B}\bigoplus_i(p_B^*\cF_{B,i} \otimes Q_{B,i}),
$$ where $\cF^{X}$ is an M-regular sheaf on $A$, $p_B: A\rightarrow B$ are surjective morphisms between abelian vareities with connected fibers, $\cF_{B,i}$  are M-regular sheaves on $B$, and $Q_{B,i}$ are torsion line bundles on $A_X$.
We can then restate the characterization of the eventual paracanonical maps.
 \begin{prop}
 The eventual paracanonical map $\varphi_X: X\rightarrow Z$  can be characterized as the map with maximal degree such that  
 \begin{itemize}
\item[(1)] $a_X$ factors as $a_X: X\xrightarrow{\varphi_X} Z\xrightarrow{g} A_X$;
\item[(2)] there exists a rank $1$ subsheaf $\cL$ of $\varphi_{X*}\omega_X$ such that $\cF^X\hookrightarrow g_*\cL$.
 \end{itemize}
 \end{prop}
By this proposition, we can give a complete description of the eventual paracanonical maps of surfaces.
 \begin{theo}\label{surface}Let $X$ be a smooth projetive surface of maximal Albanese dimension and $\chi(X, \omega_X)>0$. One of the following occurs
 \begin{itemize}
 \item[(1)] $\varphi_X$ is birational;
 \item[(2)] $\deg\varphi_X=2$ and $\chi(Z, \omega_Z)=0$;
 \item[(3)] $\varphi_X$ is a birational $(\mathbb{Z}/2\mathbb{Z})^2$-cover and $Z$ is birational to $A_X$.
 \end{itemize}
 In particular, if $\deg\varphi_X\geq 2$, the eventual paracanonical image $Z$ can not be of general type.
 \end{theo}
 
The structures of the eventual paracanonical maps in higher dimensions are much more complicated. One of the reason is that $A_X$ could admit many fibrations  in higher dimensions and hence the above decomposition formula becomes quite complicated.   Nevertheless,  we manage to  understand the eventual paracanonical maps in dimension $3$.  
 \begin{theo}\label{threefold}
 Let $X$ be a smooth projetive threefold of maximal Albanese dimension and $\chi(X, \omega_X)>0$. Then
 \begin{itemize}
 \item[(1)] $\deg\varphi_X\leq 8$ and when the equality holds, $\varphi_X$ is birationally a  $(\mathbb{Z}/2\mathbb{Z})^3$-cover and $Z$ is birational to the Abelian threefold $A_X$;
 \item[(2)] if $\chi(Z, \omega_Z)=\chi(X, \omega_X)>0$ and $\deg\varphi_X\geq 2$, then $\varphi_X$  is a  birational $G$-cover, where $G=\mathbb{Z}/3\mathbb{Z}$, or $\mathbb{Z}/2\mathbb{Z}$, or $(\mathbb{Z}/2\mathbb{Z})^2$. 
 \end{itemize}
 \end{theo}
 
 Note that Theorem \ref{surface} and Theorem \ref{threefold} answer partially the question raised by Barja, Pardini, and Stoppino.

 \addtocontents{toc}{\SkipTocEntry}
\subsection*{Acknowledgements}
We thank Miguel \'{A}ngel Barja,   Mart\'{\i} Lahoz, and Zhiyu Tian for useful conversations and comments.

 \section{Notations and Preliminaries}
 
We work over complex number field $\mathbb{C}$. A projective variety is a reduced irreducible projective scheme. Let $G$ be a finite group. We call a generically finite morphism $f: X\rightarrow Y$ between two smooth varieties a {\em birational $G$-cover} if  $G$ is a subgroup of the birational automorphism group $\mathrm{Bir}(X)$ of $X$ and $f$ is birational equivalent to the quotient $X\dashrightarrow X/G$. In other words, $f$ is a birational $G$-cover if the function field extension $\mathbb{C}(X)/\mathbb{C}(Y)$ is a Galois extension with Galois group $G$.  Note that if $f$ is a birational $G$-cover, after birational modifications of $X$, we may assume that $G$ acts faithfully on $X$.  \\

Let $A$ be an abelian variety. We will always denote by $\PA$ the dual abelian variety $\Pic^0(A)$.\\ 

A morphism $a: X\rightarrow A$ from a smooth projective variety to an abelian variety is said to be primitive if for any non-trivial $P\in \PA$, $a^*P$ is still non-trivial. The Albanese morphism $a_X: X\rightarrow A_X$ is always primitive. Let $A_X\rightarrow B$ be a fibration between abelian varieties. The induced morphism $X\rightarrow B$ is primitive.  

 We say $a$ is generating if $a(X)$ generates $A$. Assume $Y$ is a smooth projective variety and $A\rightarrow A_Y$ is an \'etale cover. Then the induced morphism $X:=Y\times_{A_Y}A\rightarrow A$ may not be the Albanese morphism of $X$ but is always primitive and generating. In this article, a morphism to an abelian variety is always primitive and generating, unless otherwise stated.\\
 
 Let $X$ be a smooth projective variety and let $a: X\rightarrow A$ be a morphism to an abelian variety. For a   cohorent sheaf $\cQ$ on $X$, define $h_{a}^0(\cQ)$ to be $\dim H^0(X, \cQ\otimes Q)$ for $Q\in \Pic^0(A)$ general. In particular, if $a$ is generically finite onto its image, by generic vanishing theory (see \cite{GL}), $h_{a}^0(\omega_X)=\chi(X, \omega_X)$.\\

Let $f: X\rightarrow Y$ be a generically finite and surjective morphism. We   denote by $\deg f$ or simply $\deg(X/Y)$ the degree of $f$. We say that $f$ is {\em minimal} if the field extension $\mathbb{C}(X)/\mathbb{C}(Y)$ is minimal, namely there is no intermediate field extension between  $\mathbb{C}(X)$ and $\mathbb{C}(Y)$.\\

Let $f: X\rightarrow Y$ be a generically finite  and surjective morphism between smooth projective varieties. Let $a: Y\rightarrow A$ be a morphism to an abelian variety such that $Y$ is generically finite  onto its image  and $a\circ f: X\rightarrow A$ is primitive and generating. We call the morphisms  $X_1\xrightarrow{f_1} Y_1\xrightarrow{a_1}A_1$ between smooth varieties {\em  a sub-representative  of $X\xrightarrow{f} Y\xrightarrow{a}A$} if $a_1\circ f_1$ is a primitive and generating morphism to the abelian variety $A_1$ and there are generically finite and surjective  morphisms  $t_X: X\rightarrow X_1$,   $t_Y: Y\rightarrow Y_1$ and $t: A\rightarrow A_1$ with a commutative diagram:
\begin{eqnarray}\label{subrepresentative}
\xymatrix{
X\ar[d]^{t_X}\ar[r]^f & Y\ar[d]^{t_Y}\ar[r]^{a} & A\ar[d]_t^{\mathrm{isogeny}}\\
X_1\ar[r]^{f_1} & Y_1\ar[r]^{a_1} & A_1}
\end{eqnarray}
such that the main component of $X_1\times_{Y_1}Y$ is irreducible and hence dominated by $X$. We often says that $a_1\circ f_1$ is a sub-representative of $a\circ f$.

Furthermore, we call  $a_1\circ f_1$  {\em  a representative of $a\circ f$} if $X$ is moreover birational to the main component of $X_1\times_{Y_1}Y$.

If $\chi(X, \omega_X)>0$, we call $a_1\circ f_1$ {\em a minimal representative } of $a\circ f$ if $a_1\circ f_1$ is a representative of $a\circ f$, $X_1$ is of general type, and $\deg(X_1/a_1\circ f_{1}(X_1))$ is minimal among all representatives.\\

 \begin{lemm}\label{representative} Let $a_1\circ f_1$ be a sub-representative of $a\circ f$ as in (\ref{subrepresentative}). Consider the canonical splittings $f_*\omega_X=\omega_Y\oplus\cQ$ and $f_{1*}\omega_{X_1}=\omega_{Y_1}\oplus \cQ_1$. Then $\cQ_1$ is a natural sub-sheaf of $t_{Y*}\cQ$.
\end{lemm}
 \begin{proof}
 We may assume that  $a_1\circ f_1$ is a representative of $a\circ f$ , otherwise we may replace $X$ by a smooth model of the main component of $X_1\times_{Y_1}Y$. 
 
 Let $\hat{X}_1$ be the normalization of $X_1$ in the Galois hull of $\mathbb{C}(X_1)/\mathbb{C}(Y_1)$. after birational modifications, we  assume that $\hat{X}_1\rightarrow Y_1$ is a birational $G$-cover and $\hat{X}_1\rightarrow X_1$ is a birational $H$-cover, where $H$ is a subgroup of $G$.   
 
 Since $f_1$ is a representative, the intersection of $\mathbb{C}(X_1)\cap\mathbb{C}(Y)=\mathbb{C}(Y_1)$, as subfields of $\mathbb{C}(X)$. Then by Galois theory, $\deg(\mathbb{C}(\hat{X}_1)\mathbb{C}(X)/\mathbb{C}(X))=\deg(\mathbb{C}(\hat{X}_1)/\mathbb{C}(X_1))$. Hence 
 $\hat{X}_1\times_{X_1}X$ has only one main component. Let $\hat{X}$ be  a smooth model of this main component. Then $\hat{X}\rightarrow Y$ is a birational $G$-cover and $\hat{X}\rightarrow X$ is a birational $H$-cover. We may also assume that $G$ acts faithfully on $\hat{X}$.
 
We may replace $X$ and $X_1$  by $\hat{X}/H$ and $\hat{X}_1/H$ respectively and replace $Y$ and $Y_1$ by $\hat{X}/G$ and $\hat{X}_1/G$, because these quotients have only finite quotient singularities and hence have rational singularties. 
Then we have 
\begin{eqnarray*}
\xymatrix{
\hat{X}\ar@/^2pc/[rr]^g\ar[r]\ar[d]^{\hat{t}} & X\ar[d]\ar[r]^f & Y\ar[d]^{t_Y}\\
\hat{X}_1\ar@/_2pc/[rr]^h\ar[r] & X_1\ar[r]^{f_1} & Y_1.}
\end{eqnarray*}

Consider the canonical splittings $g_*\omega_{\hat{X}}=\omega_Y\oplus \hat{\cQ}$ and $g_*\omega_{\hat{X}_1}=\omega_{Y_1}\oplus \hat{\cQ}_1$. Then $\hat{\cQ}$  (resp.  $\hat{\cQ}_1$) is just the direct sum of the non-trivial representations of the $G$-sheaf $g_*\omega_{\hat{X}}$ (resp. $g_*\omega_{\hat{X}_1}$). Hence $K_{\hat{X}/\hat{X}_1}$ gives us a natural inclusion $\hat{\cQ}_1\hookrightarrow t_{Y*}\hat{\cQ}$. We then take the $H$-invariant part of these two sheaves and get $\cQ_1\hookrightarrow t_{Y*}\cQ$.  We finish the proof of this lemma.

 \end{proof}

Assume that $f: Y\rightarrow Z$ is a generically finite and surjective morphism between smooth varieties, $\cM$ is a direct summand of $f_*\omega_Y$, and $g: Z\rightarrow A$ is a generically finite morphism onto its image.  Recall the {\em decomposition theorem} \cite[Theorem 1.1]{CJ}: 
\begin{eqnarray}\label{decomposition}g_*\cM=\cF^{\cM}\bigoplus_{p_B:A_X\twoheadrightarrow B}\bigoplus_i(p_B^*\cF_{B,i}^{\cM}\otimes Q_{B,i}),
\end{eqnarray} where $\cF^{\cM}$ is an M-regular sheaf on $A$, $p_B: A\rightarrow B$ are surjective morphisms between abelian vareities with connected fibers, $\cF_{B,i}^{\cM}$ are M-regular sheaves on $B$, and $Q_{B,i}$ are torsion line bundles on $A$.

We denote by $\cC_{\cM}$ the set of quotients $p_B: A\rightarrow B$   such that some $\cF_{B,i}^{\cM}$ is nontrivial.  If $\cM=\omega_Z$, we will simply denote by $\cF^Z$ the M-regular part and similarly $\cF_{B,i}^{Z}=\cF_{B,i}^{\omega_Z}$ and $\cC_{Z}=\cC_{\omega_Z}$.

Considering the eventual paracanonical map $a_X: X\xrightarrow{\varphi_X}Z\xrightarrow{g} A_X $, we always write $\varphi_{X*}\omega_X=\omega_Z\oplus\cQ$. Then if $\chi(X, \omega_X)=\chi(Z, \omega_Z)$, we have $\cF^Z=\cF^X$ and if $\chi(Z, \omega_Z)=0$, we have $\cF^Z=0$ and $\cF^{\cQ}=\cF^X$.\\

For an  \'etale cover $\pi: \tilde{A}_X\rightarrow A_X$, we  denote by $\tilde{a}_X: \tilde{X}\xrightarrow{\tilde{\varphi}_X} \tilde{Z}\xrightarrow{\tilde{g}} \tilde{A}_X$ the corresponding base change. In order to simplify notations, we will still denote by $p_B$ the Stein factorization of $\tilde{A}_X\rightarrow B$.

We always consider the following commutative diagram:
\begin{eqnarray}\label{setup}
\xymatrix{
\tilde{X}\ar[r]\ar[d] & \tilde{Z}\ar[r]\ar[d] & \tilde{A}_X\ar[d]^{p_B}\\
\tilde{X}_B\ar@/_2pc/[rr]^{\tilde{a}_B}\ar[r]^{\tilde{\varphi}_B} & \tilde{Z}_B\ar[r]^{\tilde{g}_B} & B,}
\end{eqnarray}

where we may assume all varieties are smooth projective after birational modifications, the vertical morphisms are fibrations, and the horizontal morphisms are generically finite onto its images.  

Recall from \cite[Lemma 3.7]{CJ} that, if  $\pi^*Q_{C,i}$ is trivial for all $p_C\in\cC_X$ and $i$,  then 
\begin{eqnarray}\label{explanation}\tilde{a}_{B*}\omega_{\tilde{X}_B}=\bigoplus_{\scriptstyle p_C\in\cC_X  \atop \scriptstyle \PC\subset\PB }\bigoplus_ip_{B,C}^*\cF_{C,i}^X=\tilde{g}_{B*}\omega_{\tilde{Z}_B}\bigoplus_{\scriptstyle p_C\in\cC_{\cQ}  \atop \scriptstyle \PC\subset\PB } \bigoplus_{i}p_{B,C}^*\cF_{B,i}^{\cQ},
\end{eqnarray}
where $p_{B,C}: B\rightarrow C$ is the natural quotient.

Hence $p_B\in\cC_X$ if and only if for some finite abelian \'etale cover $\tilde{X}\rightarrow X$,  $\chi(\tilde{X}_B, \omega_{\tilde{X}_B})>0$. Similarly, $p_B\in\cC_{\cQ}$ if and only if for some finite abelian \'etale cover $\tilde{X}\rightarrow X$,  $\chi(\tilde{X}_B, \omega_{\tilde{X}_B})>\chi(\tilde{Z}_B, \omega_{\tilde{Z}_B})$.

\begin{lemm}\label{line}
Let $\varphi_X: X\rightarrow Z $ be the eventual paracanonical map of $X$. Then after birational modifications of $X$ and $Z$, there exists a line bundle $\cL$ on $Z$ such that $\cL$ is a subsheaf of $\varphi_{X*}\omega_X$ and $\cF^X$ is naturally a subsheaf of $g_*\cL$.
\end{lemm}
\begin{proof}
As above, let $X^{(d)}\rightarrow X$ be an abelian \'etale cover with Galois group $G$ for $d$ sufficiently large and divisible. Then the complete linear system $|K_{X^{(d)}}|$ defines a $G$-equivariant map $X^{(d)}\rightarrow Z^{(d)}\hookrightarrow \mathbb{P}:=\mathbb{P}(H^0(X^{(d)}, K_{X^{(d)}})^*)$ and $G$ acts on $Z^{(d)}$ freely. Note that $Z=Z^{(d)}/G$, hence  $\cO_{\mathbb{P}}(1)$ descends to a line bundle $\cL$ on $Z$ and $\cL$ is a subsheaf of $\varphi_{X*}\omega_X$. Moreover, since $h^0(X^{(d)}, K_{X^{(d)}})=h^0(Z^{(d)}, \cO_{\mathbb{P}}(1))$, we have 
$h^0(X, K_X\otimes P)=h^0(Z, \cL\otimes P)$ for $P\in\Pic^0(X)$ general. Thus  the image of the continuous evaluation map $$\bigoplus_{Q\in \Pic^0(A_X)\;\; \mathrm{general}}H^0(Z, \varphi_{X*}\omega_X\otimes Q)\otimes Q^{-1}\rightarrow \varphi_{X*}\omega_X$$ is contained in $\cL$.
Moreover, $H^0(Z, \varphi_{X*}\omega_X\otimes Q)=H^0(A_X, a_{X*}\omega_X\otimes Q)$ for all $Q$. Hence the image of of the continuous evaluation map $$c_X: \bigoplus_{Q\in \Pic^0(A_X)\;\; \mathrm{general}}H^0(A_X, a_{X*}\omega_X\otimes Q)\otimes Q^{-1}\rightarrow a_{X*}\omega_X$$ is contained in $g_*\cL$. By \cite[Proposition 2.13]{PP}, the image of $c_X$ is $\cF^X$. Hence $\cF^X$ is contained in $g_*\cL$. 
\end{proof}

 \begin{coro}\label{intermediate}
 Assume that the eventual paracanonical map of $X$ factors through  a smooth projective variety $Y$: $$\varphi_X: X\xrightarrow{\varphi_1} Y\xrightarrow{\varphi_2} Z.$$ Then either $\chi(X, \omega_X)=\chi(Y, \omega_Y)>0$ or $\chi(Y, \omega_Y)=0$.
 \end{coro}
 \begin{proof}
 If $\chi(X, \omega_X)=\chi(Z, \omega_Z)$, then we have $\chi(X, \omega_X)=\chi(Y, \omega_Y)$. 
 
 We just need to deal with the case $\chi(Z, \omega_Z)=0$. We  write $\varphi_{1*}\omega_X=\omega_Y\oplus \cM_1$ and $\varphi_{2*}\omega_Y=\omega_Z\oplus \cM_2$. By Lemma \ref{line}, after birational modifications, there exists a line bundle $\cL$ on $Z$ such that $\cL$ is a subsheaf of $\varphi_{X*}\omega_X=\omega_Z\oplus \varphi_{2*}\cM_1\oplus\cM_2$ and $h^0_g(Z,\cL)=\chi(X, \omega_X)$. Note that since $\chi(Z, \omega_Z)=0$, the natural morphism $\cL\rightarrow \omega_Z$ is trivial.  If the natural morphism of sheaves $\cL\rightarrow \varphi_{2*}\cM_1$ is non-trivial, then $\chi(Y, \omega_Y)=h^0_{g}(Z, \varphi_{2*}\cM_1)\geq h^0_g(Z, \cL)=\chi(X, \omega_X)$. If the morphism $\cL\rightarrow \varphi_{2*}\cM_1$ is trivial, then $\cL$ is a subsheaf of $\cM_2$, then  $a_{Y*}\omega_Y$, as a direct summand of $a_{X*}\omega_X$, does not contain any M-regular direct summand. Hence $\chi(Y, \omega_Y)=0$.
 \end{proof}
 
 The following corollary gives a sufficient condition for $\varphi_X$ to be birational.
 \begin{coro}\label{criteria}
 Let $a_X: X\rightarrow A_X$ be the Albanese morphism of a variety of maximal Albanese dimension. Let $\cF^X$ be the M-regular part of $a_{X*}\omega_X$. Assume that $\deg(X/a_X(X))<2\rank\cF^X$, the eventual paracanonical map of $X$ is birational. 
 \end{coro}
 \begin{proof}
 If $\varphi_X$ is not birational, then $a_X$ factors through $\varphi_X$. Hence $\deg(X/a_X(X))\geq 2\deg(Z/a_X(X))$. By Lemma \ref{line}, $\rank\cF^X\leq \rank g_*\cL=\deg(Z/a_X(X))$ and we get a contradiction.
 
 \end{proof}

The sets  $\cC_X$ and $\cC_{\cQ}$ play important a role in the study of $\varphi_X$.
 
 \begin{coro}\label{fiber}Let $X$ be a smooth projective variety of maximal Albanese dimension and $\chi(X, \omega_X)>0$. If  $\deg \varphi_X>1$, $\cC_X\neq \emptyset$ and hence $a_X(X)$ is fibred by tori. If $\deg\varphi_X>2$, $\cC_{\cQ}\neq \emptyset$.
\end{coro}
\begin{proof} 
By the assumption that $\deg\varphi_X>1$, $g_*\cL$ is a proper subsheaf of $a_{X*}\omega_X$. By the above lemma, we conclude that the second part in the decomposition of $a_{X*}\omega_X$ is non-trivial, namely there exists $p_B: A_X\rightarrow B$ such that $p_B^*\cF_B^X\otimes Q_B$ is a direct summand of $a_{X*}\omega_X$. Thus, $\cC_X\neq\emptyset$. Moreover, $a_{X*}\omega_X$ is a torsion-free sheaf supported on $a_X(X)$. Hence $a_X(X)$ is fibred by the kernel of $p_B$.

Similarly, if $\deg\varphi_X>2$, $\rank\cQ>1$. Thus $g_*\cQ$ is not M-regular and $\cC_{\cQ}\neq \emptyset$.
 \end{proof}
 \begin{lemm}\label{complementary} Let $A_X\rightarrow B$ be a non-trivial quotient with connected kernel $K$ and let $\pi: \tilde{A}_X\rightarrow A_X$ be an \'etale cover. Considering the commutative diagram (\ref{setup}),
 assume that the induced morphism $\tilde{\varphi}_B: \tilde{X}_B\rightarrow \tilde{Z}_B$ has degree $>1$. Then there exists $p_{B'}\in\cC_X$ such that the morphism $ p_{B'}\mid_K: K\rightarrow B'$ is an isogeny onto its image. Moreover, if $\chi(X, \omega_X)=\chi(Z, \omega_Z)$,   we can choose $p_{B'}\in \cC_{\cQ}$.
 \end{lemm}
 \begin{proof}
Since the argument would be the same, we may simply assume that the \'etale cover $\pi$ is identity of $A_X$ and $\deg \varphi_B>1$.
 
 We then take $t\in B$ a general point in the image of $Z_B$, $s_1,\cdots, s_m$ the pre-image of $t$ in $Z_B$ and for each $s_i$, we denote by $t_{i1},\cdots, t_{ik}$ the pre-image in $X_B$. Note that by assumption $k=\deg \varphi_B>1$. 
 
Denote by $K$ the fiber of $A_X\rightarrow B$ over $t$, $Z_i$ the fiber of $Z\rightarrow Z_{B}$ over $s_i$, and $X_{ij}$ the fiber of $X\rightarrow X_B$ over $t_{ij}$. We will denote by $X_{ij}\xrightarrow{\varphi_{ij}} Z_i\xrightarrow{g_i} K$ the morphisms between the fibers. We consider $a_{X*}\omega_X\mid_K$. By base change, 
 \begin{eqnarray}\label{iso}
 \cF^X\mid_K\oplus \bigoplus_{p_B\in\cC_X}\bigoplus_k (p_B^*\cF^X_{B,k}\otimes Q_{B,k})\mid_K&=&(a_{X*}\omega_X)\mid_K\\\nonumber
 &=&g_* \omega_Z\mid_K\oplus g_*\cQ\mid_K\\\nonumber&=&\bigoplus_{1\leq i\leq m}\bigoplus_{1\leq j\leq k}g_{i*}\varphi_{ij*}\omega_{X_{ij}}.
 \end{eqnarray}
  
Since $\varphi_X: X\rightarrow Z$ is the eventual paracanonical map of $X$, after birational modifications of $X$ and $Z$, there exists a line bundle $\cL$  of $\varphi_{X*}\omega_X$  such that $g_*\cL$ contains $\cF^X$. Then $\varphi_X^*\cL\hookrightarrow \omega_X$, we  have $\varphi_X^*\cL\mid_{X_{ij}}=\varphi_{ij}^*(\cL\mid_{Z_i})\hookrightarrow \omega_{X_{ij}}$ on each $X_{ij}$. In particular, $ \cL\mid_{Z_i}\hookrightarrow\varphi_{ij*}\omega_{X_{ij}}$ for each $j$. Moreover, $(g_*\cL)\mid_K\simeq \bigoplus_{1\leq i\leq m}g_{i*}(\cL\mid_{Z_i})$ by base change. 

Thus $\cF^X\mid_K\hookrightarrow \bigoplus_{1\leq i\leq m}g_{i*}\varphi_{ij*}\omega_{X_{ij}}$, for each $j$.
 
Note that $\cF^X\mid_K$ is M-regular, we compare the M-regular part of both sides of (\ref{iso}) and  conclude that, for some $p_{B'}\in\cC_X$, $(p_{B'}^*\cF^X_{B', k})\mid_K$ is M-regular. Then $p_{B'}\mid_K: K\rightarrow B'$ is an isogeny onto its image.
  
  Note that if $\chi(X, \omega_X)=\chi(Z, \omega_Z)$, then $\cL=\omega_Z$. Then $g_{i*}\omega_{Z_i}$ contains a M-regular sheaf on $K$. Since $\deg\varphi_B=k>1$, we see that $\omega_{Z_i}\hookrightarrow \cQ\mid_{Z_i}$. Hence $g_*\cQ\mid_K$ also contains non-trivial M-regular part and we can pick $p_{B'}\in \cC_{\cQ}$. 
  \end{proof}
 
 \begin{coro}\label{complementary2}
Assume that $\deg\varphi_X>1$ and $\chi(X, \omega_X)=\chi(Z, \omega_Z)$. For any $p_{B_1}\in\cC_{\cQ}$, there exists $p_{B_2}\in\cC_{\cQ}$ such that $p_{B_1}\times p_{B_2}: A_X\rightarrow B_1\times B_2$ is an isogeny onto its image.
 \end{coro}
  \begin{proof}
  By the assumption, $\cC_{\cQ}\neq \emptyset$.
   We may take an \'etale cover $\tilde{A}_X\rightarrow A_X$ such that the pull-back of all $Q_{B_1,i}$ in (\ref{decomposition}) is trivial. Let $p_{B_1}\in \cC_{\cQ}$. Then we know from (\ref{explanation})  that $\tilde{\varphi}_B: \tilde{X}_B\rightarrow \tilde{Z}_B$ has degree $>1$. Then by Lemma \ref{complementary}, there exists $p_{B_2}\in \cC_{\cQ}$ such that $p_{B_1}\times p_{B_2}: A_X\rightarrow B_1\times B_2$ is an isogeny onto its image.
  \end{proof}

Finally, we shall briefly mention the general strategy of this paper.   
\begin{rema}\label{strategy}
We are interesting about the structure of $\varphi_X$. We often take a sufficiently large and divisible \'etale cover $\pi: \tilde{A}_X\rightarrow A_X$ of abelian varieties \footnote{we can take $\pi$ to be the $m_d: A_X\rightarrow A_X$ for $d$ sufficiently large and divisible. } and consider the base changes $\pi_X:\tilde{X}\rightarrow X$, $\pi_Z:\tilde{Z}\rightarrow Z$, and $\tilde{a}_X: \tilde{X}\xrightarrow{\tilde{\varphi}_X} \tilde{Z}\xrightarrow{\tilde{g}}\tilde{A}_X$. Then we may and will assume that $\pi^*Q_{B,i}$ are trivial line bundles for all $Q_{B,i}$ appears in the decomposition formula of $a_{X*}\omega_X$ and the morphism $\tilde{\varphi}_X$ is birational equivalent to the map induced by $|K_{\tilde{X}}\otimes Q|$ for $Q\in \PA_X$ general.

Note that the structure of $\varphi_X$ is the same as the structure of $\tilde{\varphi}_X$, $\pi_Z^*\cL\hookrightarrow \tilde{\varphi}_{X*}\omega_{\tilde{X}}$, and $h_{\tilde{g}}^0(\pi_Z^*\cL)=\chi(\tilde{X},\omega_{\tilde{X}})$. Thus we can focus on the structure of $\tilde{\varphi}_X$. The advantage is that, by (\ref{explanation}), Lemma \ref{fiber},  Lemma \ref{complementary}, and Corollary \ref{complementary2}, we have several irregular fibrations $\tilde{X}\rightarrow \tilde{X}_B$ with $\tilde{X}_B$ of general type. The existence of such fibrations together with the assumption that $\cF^{\tilde{X}}=\pi^*\cF^X$ is contained in the sheaf $\tilde{g}_*(\pi_Z^*\cL)$ whose rank is $\leq \deg\tilde{g}$, give very strong restrictions on the structure of $\tilde{X}$, $\tilde{Z}$, and $\tilde{\varphi}_X$. 
\end{rema}

 \section{Surface case}
 In this section, we study the eventual paracanonical maps in dimension $2$. We already know by \cite{BPS2} that $\deg\varphi_X\leq 4$ and $\deg\varphi_X\leq 2$ if $Z$ is also of general type. We shall first show that in the latter case, $X$ is always birational to $Z$.
 
\subsection{$Z$ cannot be of general type}

\begin{prop}\label{technical} 
Let $Z$ be a smooth projective surface of maximal Albanese dimension and of general type. Let $t: X\rightarrow Z$ be a generically finite morphism between smooth projective surfaces and $\deg t>1$. Assume that $\chi(X, \omega_X)=\chi(Z, \omega_Z)$. Then there exists double covers $C_i\rightarrow B_i$ of elliptic curves with involutions $\tau_i$, where $C_i$ is a smooth projective curve of genus $>1$ and an abelian \'etale cover $\tilde{Z}\rightarrow Z$ with the induced base change $\tilde{t}:\tilde{X}\rightarrow \tilde{Z}$ such that we have the following commutative diagram:
\begin{eqnarray*}
\xymatrix{
\tilde{X}\ar[r]^{\tilde{t}}\ar[d]^{birational} & \tilde{Z}\ar[d]^{birational}\ar[r] & \tilde{A}_{Z}\ar@{=}[d] \\
\tilde{C_1\times C_2}\ar[d]^{\text{\'etale}}\ar[r] & \tilde{(C_1\times C_2)/\langle\tau_1\times \tau_2\rangle}\ar[d]^{\text{\'etale}}\ar[r] & \tilde{A}_{Z}\ar[d]\\
C_1\times C_2\ar[r] & (C_1\times C_2)/\tau_1\times\tau_2\ar[r] & B_1\times B_2,}
\end{eqnarray*}
where the two commutative diagram downstairs are Cartesian.
\end{prop}

We will see a different proof of this proposition later, which is quite short (see Remark \ref{relation} and Remark \ref{easyproof}) but depends heavily on a previous result in \cite{CDJ}. We insist to write down this technical proof because it shed some light on higher dimensional cases.
\begin{proof}
By \cite[Proposition 3.3]{BPS2}, $\deg t=2$.  
 
Let $\rho: X\xrightarrow{t} Z\xrightarrow{a_Z} A_Z$ be the composition of morphisms and denote $t_*\omega_X=\omega_Z\oplus\cQ$. We then have $\rho_*\omega_X=a_{Z*}\omega_X\oplus a_{Z*}\cQ$. Since $\chi(X, \omega_X)=\chi(Z, \omega_Z)$, $a_{Z*}\cQ$ is the direct sum of pull-back of M-regular sheaves on lower dimensional abelian varieties by the decomposition (\ref{decomposition}). We then take an \'etale cover of $\tilde{A_Z}\rightarrow A_Z$  with induced morphisms $$\tilde{X}\xrightarrow{\tilde{t}} \tilde{Z}\xrightarrow{\tilde{g}} \tilde{A}_Z,$$ and the splitting $\tilde{t}_*\omega_{\tilde{X}}=\omega_{\tilde{Z}}\oplus\tilde{\cQ}$ such that 
\begin{eqnarray}\label{dec-surface}\tilde{g_*}\tilde{\cQ}=\bigoplus_{p_B: \tilde{A}_Z\rightarrow B}p_B^*\cF_B.
\end{eqnarray}

\begin{claim}
The surface $\tilde{Z}$ does not admit a fibration to a smooth curve $C$ of genus $\geq 2$.
\end{claim} 
\begin{proof}[Proof of the claim]
If the claim is not true, there exists a fibration $h_1: \tilde{Z}\rightarrow C$ and denote by $h_2: \tilde{X}\rightarrow C$ be the induced morphism. Then for $Q\in \Pic^0(A_Z)$ general, we have an inclusion of vector bundles $h_{1*}(\omega_{\tilde{Z}}\otimes Q)\hookrightarrow  h_{2*}(\omega_{\tilde{X}}\otimes Q)$. Since $\tilde{t}$ is not birational, the two sheaves have different rank. Moreover, by Fujita's theorem (see for instance \cite{Vie}), both $h_{1*}(\omega_{\tilde{Z}/C}\otimes Q)$  and $h_{2*}(\omega_{\tilde{X}/C}\otimes Q)$ are nef. Then by Riemann-Roch, the global sections of these two sheaves are different. This is a contradiction to the assumption that $\chi(X, \omega_X)=\chi(Z, \omega_Z)$.
\end{proof}

Hence any $B$ appearing in (\ref{dec-surface}) is an elliptic curve, and for any $p_B$ in (\ref{dec-surface}), the induced morphism $\tilde{Z}\rightarrow B$ is a fibration over the elliptic curve $B$.
 Let $\tilde{X}\xrightarrow{q} \tilde{X}_B\xrightarrow{t_B} B$ be the Stein factorization of the morphism $\tilde{X}\rightarrow B$. 
 By (\ref{explanation}), we know  that $t_{B*}\omega_{\tilde{X}_B}=\omega_B\oplus \cF_B$.

In particular, $\deg t_B\geq 2$.  Since $\deg \tilde{t}=2$ and $\tilde{t}$ factors through $\tilde{X}\rightarrow \tilde{X}_B\times_B\tilde{Z}\rightarrow \tilde{Z}$, we conclude that $t_B$ is also a double cover and $\cF_B$ is a line bundle on $B$.

We denote by $m=\deg a_Z=\deg \tilde{g}=\rank \tilde{g}_*\tilde{\cQ}>1$. Then we see that there are at least two fibrations $p_{B_1}$ and $p_{B_2}$ in the decomposition (\ref{dec-surface}). Let $C_i=X_{B_i}$, $\tau_i$ the involutions of the double covers $C_i\rightarrow B_i$, and $\tau$ the involution of $\tilde{t}$. Then we have the $\mathbb{Z}/2\mathbb{Z}$-equivariant morphism $\tilde{X}\rightarrow C_1\times C_2$ with the commutative diagram:
\begin{eqnarray*}
\xymatrix{
\tilde{X}\ar[d]\ar[r]^{\tilde{t}} & \tilde{Z}\ar[d]\ar[r] & \tilde{A}_{Z}\ar[d]^{\pi}\\
C_1\times C_2\ar[r] & (C_1\times C_2)/\langle\tau_1\times\tau_2\rangle\ar[r] & B_1\times B_2.}
\end{eqnarray*} 

We claim that $\cC_{\tilde{\cQ}}=\{p_{B_1}, p_{B_2}\}$. We argue by contradiction. Assume $p_{B_3}\in \cC_{\tilde{\cQ}}$. Then we have a double cover $C_3\rightarrow B_3$ with the corresponding involution $\tau_3$. 
Let $\tilde{C_1\times C_2}$ and $\tilde{(C_1\times C_2)/\langle\tau_1\times \tau_2\rangle}$ be respectively the \'etale cover of $C_1\times C_2$ and $(C_1\times C_2)/\langle\tau_1\times \tau_2\rangle$ induced by $\pi$. 
The induced morphism $\tilde{C_1\times C_2}\rightarrow B_3$ is a  fibration by \cite[Lemma 5.3]{CDJ}. Let $X_1$ be a desingularization of the main component of $(\tilde{C_1\times C_2})\times_{B_3}C_3$.    
We shall still denote by $\tau_i$ the natural lift of $\tau_i$-action of $C_i$ on $X_1$. 
Then the natural morphism $\tilde{X}\rightarrow X_1$ is a $(\mathbb{Z}/2\mathbb{Z})$-equivaraint morphism with the $\tau$-action on $\tilde{X}$ and the $(\tau_1\times\tau_2\times\tau_3)$-action on $X_1$. Obviously, the natural morphism $a_1: X_1\rightarrow \tilde{A}_Z$ is  a birational $G:=(\mathbb{Z}/2\mathbb{Z})^3$-cover, 
where $\tau_1$, $\tau_2$, and $\tau_3$ are the generator of $G$. 
Let $\chi_i\in G^*$ be the character such that $\chi_i(\tau_i)=-1$ and $\chi_i(\tau_j)=1$ for $j\neq i$. 
 Let $Z_1=X_1/\langle \tau_1\times\tau_2\times\tau_3\rangle$ and consider the commutative diagram:
\begin{eqnarray*}
\xymatrix{
\tilde{X}\ar@/^2pc/[rr]^{\tilde{a}}\ar[d]\ar[r]^{\tilde{t}} & \tilde{Z}\ar[d]\ar[r]^{\tilde{g}} & \tilde{A}_{Z}\ar@{=}[d]\\
X_1\ar@/_2pc/[rr]^{a_1}\ar[r]^{t_1} &  Z_1\ar[r]^{g_1} & \tilde{A}_{Z}.}
\end{eqnarray*}
Let us write $t_{1*}\omega_{X_1}=\omega_{Z_1}\oplus\cQ_1$. 
We have a decomposition $a_{1*}\omega_{X_1}=\bigoplus_{\chi\in G^*}\cL_{\chi}$ as the eigensheaves of the action of $G$.
For $\{i, j, k\}=\{1,2,3\}$, the direct summand $p_{B_i}^*\cF_{B_i}$ is $\tau_j$ and $\tau_k$ invariant, hence  it is isomorphic to $\cL_{\chi_i}$.
Moreover, the sheaf $g_{1*}\omega_{Z_1}$ is the $(\tau_1\times\tau_2\times\tau_3)$-invariant part and thus $$g_{1*}\cQ_1=\cL_{\chi_1}\oplus \cL_{\chi_2}\oplus\cL_{\chi_3}\oplus\cL_{\chi_1\chi_2\chi_3}.$$ 

We claim that $\cL_{\chi_1\chi_2\chi_3}$ is M-regular on $\tilde{A}_Z$. Otherwise, $\cL_{\chi_1\chi_2\chi_3}$ is a line bundle pulled back from an elliptic curve by the decomposition (\ref{dec-surface}). 
One the other hand, the sheaf ${a_1}_* \mathcal{O}_{\tilde{X}}=\cHom({a_1}_* \omega_{\tilde{X}}, \mathcal{O}_{\tilde{A}_Z})$ is an algebra, whose multiplication is compatible with the $G$-action. 
By considering the multiplication we get that $\cL_{\chi_j\chi_k}\hookrightarrow \cL_{\chi_i}\cL_{\chi_1\chi_2\chi_3}$ for any $\{i, j, k\}={1, 2, 3}$. Moreover, by the structure of $X_1$, we have $\cL_{\chi_j\chi_k}=\cL_{\chi_j}\otimes\cL_{\chi_k}$, for any $1\leq j<k\leq 3$. Hence  $\cL_{\chi_1}\cL_{\chi_2}\cL_{\chi_3}\hookrightarrow\cL_{\chi_1\chi_2\chi_3}^{\otimes 3} $, which is a contradiction. 

We then get a contradiction as following. Note that, being  respectively the anti-invariant parts of $\tilde{a}_*\omega_{\tilde{X}}$ and $a_{1*}\omega_{X_1}$,  $g_{1*}\cQ\hookrightarrow \tilde{g}_*\tilde{\cQ}$. However, $g_{1*}\cQ$ contains the M-regular sheaf $\cL_{\chi_1\chi_2\chi_3}$, this is not compatible with $(\ref{dec-surface})$.

We can now conclude the proof. As $\mathcal{C}_{\tilde{\cQ}}=\{p_{B_1}, p_{B_2}\}$. We see by $(\ref{dec-surface})$ that $\deg a_Z=\deg \tilde{g}=2$. Hence $\tilde{Z}$ is birational to the \'etale cover of $(C_1\times C_2)/\langle\tau_1\times\tau_2\rangle$ induced by $\pi$ and similarly $\tilde{X}$ is birational to the \'etale cover of $C_1\times C_2$ induced by $\pi$. 
 \end{proof}
\begin{rema}\label{relation}
Assume that $t: X\rightarrow Z$ is a generically finite and surjective morphism of degree $>1$ between smooth varieties of maximal Albanese dimension. Then the hypothesis that $\chi(X, \omega_X)=\chi(Z, \omega_Z)>0$ is related to varieties $Y$ of general type, of  maximal Albanese dimension, and of $\chi(Y, \omega_Y)=0$, which have been studied in \cite{CDJ, CJ}.

Indeed, assume moreover that $t$ is a birational $G$-cover and we have a $G$-cover $s: C\rightarrow E$ from a smooth curve of genus $\geq 2$ to an elliptic curve. Write $t_*\omega_X=\omega_Z\oplus \cQ_1$ and $s_*\omega_C=\cO_E\oplus \cQ_2$. Note that,by the assumption that $\chi(X, \omega_X)=\chi(Z, \omega_Z)$, $a_{Z*}\cQ_1$ has no M-regular direct summand.

Let $Y$ be a smooth model of the diagonal quotient $(X\times C)/G$. We claim that \\
{\em $Y$ is  of general type, of  maximal Albanese dimension, and  $\chi(Y, \omega_Y)=0$.}\\

Consider the following commutative  diagram
\begin{eqnarray*}
\xymatrix{& Y\ar[d]^{\epsilon}\ar[drr]^{\varphi}\\
X\times C\ar[r] \ar@/_2pc/[rrr]^{\rho}& (X\times C)/G\ar[r] & Z\times E\ar[r]_{a_Z\times \mathrm{id}_E} & A_Z\times E.}
\end{eqnarray*}
Since $(X\times C)/G$ has quotient singularities, $\epsilon_*\omega_Y=\omega_{(X\times C)/G}$. Hence 
$$\varphi_*\omega_Y=(\rho_*\omega_{X\times C})^G=(a_Z\times\mathrm{Id}_E)_*\omega_{Z\times E}\bigoplus (a_{Z*}\cQ_1\boxtimes \cQ_2)^G.$$

We then deduce that the M-regular part $\cF^Y=0$. Hence $\chi(Y, \omega_Y)=0$. Moreover,  $Y$ has a fibration to $Z$, whose general fiber is $C$.  Iitaka's $C_{nm}$ conjecture about subadditivity of Kodaira dimensions is solved when the base is general type (see \cite{V}). Thus $Y$ is also of general type and we finish the proof of the claim.\end{rema}

\begin{rema}\label{easyproof}
We can apply this observation to give a short proof of Proposition  \ref{technical}. Indeed, under the assumption of Proposition \ref{technical}, $t$ is a birational $\mathbb{Z}/2\mathbb{Z}$-cover. We then take a bi-eliptic curve $C\rightarrow E$ with $g(C)\geq 2$ and let $Y$ be a smooth model of the diagonal quotient $(X\times C)/(\mathbb{Z}/2\mathbb{Z})$. As we have seen that $Y$ is of general type, of maximal Albanese dimension, and $\chi(Y, \omega_Y)=0$. By the main result of \cite{CDJ}, we know that a finite abelian \'etale cover of  $Y$ is birational to an Ein-Lazarsfeld threefold $(C_1\times C_2\times C_3)/\langle\tau_1\times \tau_2\times\tau_3\rangle$, where $C_i$ are bi-elliptic curves with bi-elliptic involution $\tau_i$. It is then easy to check that a finite abelian \'etale cover of $X$ (resp. $Z$) is birational to $C_i\times C_j$ (resp. $(C_i\times C_j)/\langle\tau_i\times\tau_j\rangle$) for some $i,j\in\{1, 2, 3\}$.
 \end{rema}  

\begin{prop}\label{non-gt}Assume that $\dim X=2$, then either $\varphi_X$ is birational or $\chi(Z, \omega_Z)=0$.
\end{prop}

\begin{proof}
We just need to prove that if $\chi(X, \omega_X)=\chi(Z, \omega_Z)>0$,  $\varphi_X$ is birational. We argue by contradiction.

 Assume that $\varphi_X: X\rightarrow Z$ is the eventual paracanonical map, $\deg \varphi_X>1$, and $\chi(X, \omega_X)=\chi(Z, \omega_Z)>0$. Then $\dim |K_X\otimes Q|=\dim |K_Z\otimes Q|$ for $Q\in\Pic^0(A_X)$  general and the eventual paracanonical map of $Z$ is birational.

On the other hand, we  apply Proposition \ref{technical} to conclude that $g: Z\rightarrow A_X$ is of degree $2$. By \cite[Lemma 5.1]{BPS2}, the eventual paracanonical map of $Z$ is birational equivalent to $g$, which is a contradiction.
 \end{proof}
 \subsection{Complete descriptions of $\varphi_X$ of surfaces}
Barja, Pardini, Stoppino constructed surfaces $X$ such that $\deg \varphi_X=2$ or $4$. We will show in the following that $\deg \varphi=3$ can never happen and the structure of $
\varphi_X$ is quite restrictive when $\deg \varphi_X=4$.
 
\begin{lemm}\label{image}Let $X$ be a surface. If $\deg \varphi_X>2$,  $g$ is birational.
\end{lemm}
\begin{proof}
If $\kappa(Z)=0$, then $Z$ is birational to an abelian surface and we may assume that $Z$ is an abelian surface. By the universal property of $a_X$, we conclude that $Z\simeq A_X$ and $g$ is an isomorphism.
  By Proposition \ref{non-gt}, $\kappa(Z)\leq 1$.  Hence we just need to rule out the case that $\kappa(Z)=1$.  We argue by contradiction. Assume that $\deg \varphi_X>2$ and $\kappa(Z)=1$. We may assume that $Z$ is minimal. In this case, we can  take a finite abelian \'etale cover $\pi: \tilde{Z}\rightarrow Z$ such that $\tilde{Z}$ is a smooth elliptic fibration over a genus $\geq 2$ curve $C$ and $\tilde{Z}$ is embedded in $A_{\tilde{Z}}$. Let $\tilde{X}:=X\times_Z\tilde{Z}$ be the induced \'etale cover. Note that, since $a_X$ factors through $\varphi_X$, $\tilde{X}$ is connected and $\deg(\tilde{X}/\tilde{Z})=\deg \varphi_X>2$.
  
 We then consider the commutative diagram:
\begin{eqnarray*}
\xymatrix{
\tilde{X}\ar@/^2pc/[rr]^{f}\ar[d]^{q_{C'}} \ar[r]^{\tilde{\varphi}_X} & \tilde{Z} \ar[d]^{q_C}\ar@{^{(}->}[r] & A_{\tilde{Z}}\ar[d] \\
C'\ar[r]^t & C\ar@{^{(}->}[r]  & JC,}
\end{eqnarray*} 
where $q_{C'}$ is the Stein factorization of $\tilde{X}\rightarrow C$.
We denote by $k=\deg \varphi_X=\deg \tilde{\varphi}_X$.

 We consider $\tilde{\varphi}_{X*}\omega_{\tilde{X}}=\omega_{\tilde{Z}}\oplus \cM$. Let $\cL$ be the torsion-free rank $1$ sheaf in Lemma \ref{line}. 
Then $\pi^*\cL\hookrightarrow \tilde{\varphi}_{X*}\omega_{\tilde{X}}$. Since $\kappa(\tilde{Z})=1$, we have $\pi^*\cL\hookrightarrow \cM$. Moreover, since $\pi$ is an \'etale cover, \begin{eqnarray}\label{equal1}\chi(\tilde{X}, \omega_{\tilde{X}})=h^0_{a_{\tilde{Z}}}(\pi^*\cL).\end{eqnarray}
We then apply the decomposition theorem to $f_*\omega_{\tilde{X}}=\tilde{\varphi}_{X*}\omega_{\tilde{X}}$.  By (\ref{equal1}), we conclude that $\pi^*\cL=\cF^{\tilde{X}}$ is the M-regular part of $f_*\omega_{\tilde{X}}$. Since $\rank \cM=k-1>1$, we conclude that $\cM=\cL\oplus \cM'$, where $\cM'\neq 0$ is  not M-regular.
 
 Note that $\cM'$ is supported on $\tilde{Z}$ and hence after further abelian \'etale cover of $\tilde{Z}$, we may write $\cM'=q_C^*\cM_C$.
By Koll\'ar's theorem \cite[Proposition 7.6]{K1}, $\omega_{C'}=R^1q_{C'*}\omega_{\tilde{X}}$.  Then $t_*\omega_{C'}=t_*R^1q_{C'*}\omega_{\tilde{X}}=R^1q_{C*}\tilde{\varphi}_{X*}\omega_{\tilde{X}}\supseteq \omega_C\oplus \cM_C$. Thus $\deg t\geq k-1$. On the other hand, $\tilde{X}$ is of general type, hence a general fiber of $q_{C'}$
is a curve of genus $\geq 2$. Thus $k=\deg \tilde{\varphi}_X\geq 2\deg t\geq 2(k-1)$ and hence $k\leq 2$, which is a contradiction.

\end{proof}

\begin{theo}\label{surface}
Let $X$ be a surface such that the eventual paracanonical $\varphi_X$ is not birational. Then either $\deg \varphi_X=2$ and $Z$ can be any surface of maximal Albanese dimension with $0\leq \kappa(Z)\leq 1$ or $\varphi_X$ is  a birational $(\mathbb{Z}/2\mathbb{Z})^2$-cover, $Z$ is birational to $A_X$, and an \'etale cover of $X$ is birational to a product of two bielliptic curves.
\end{theo}
\begin{proof}
In \cite[Example 5.2]{BPS2}, the authors prove that every variety of maximal Albanese dimension $Z$ with $\chi(Z, \omega_Z)=0$ occurs as the eventual paracanonical image for some $X$ with $\deg \varphi_X=2$. Note that for surfaces $Z$ of maximal Albanese dimension, $\chi(Z, \omega_Z)=0$ is equivalent to say that $0\leq \kappa(Z)\leq 1$.

We see in Lemma \ref{image} that if $\deg \varphi_X=3$ or $4$, then $\kappa(Z)=0$ and hence we may assume in these cases that $Z=A_X$. We then apply the decomposition theorem and note that in this case $\cF^X=\cL$: 
\begin{eqnarray*}a_{X*}\omega_X=\cO_{A_X}\oplus \cL\oplus \cM,
\end{eqnarray*}
where $\cM$ is the direct sum of the direct summands which are not M-regular.

If $\deg\varphi_X=3$, then $\rank \cM=1$. Hence $\cM$ can be wrote as $p_B^*\cL_B\otimes Q_B$ for some line bundle $\cL_B$ on an elliptic curve $B$. Then after an \'etale cover $\tilde{A}_X\rightarrow A_X$, the base-change $\tilde{X}:=X\times_{A_X}\tilde{X}$ has a fibration $\tilde{X}\rightarrow Z_B\rightarrow B$ such that $\deg(Z_B/B)=2$. This is impossible. Hence $\deg \varphi_X=3$ can never happen.

If $\deg \varphi_X=4$, we argue that $\cM$ is the direct sum of  two direct summands $(p_B^*\cL_B\otimes Q_B)\oplus (p_C^*\cL_C\otimes Q_C)$. Otherwise,  $\cM=p_B^*\cF_B\otimes Q_B$ or $\cM=(p_B^*\cL_B\otimes Q_1)\oplus (p_B^*\cL_B'\otimes Q_2)$. In both cases, after an \'etale cover of $A_X$, the induced \'etale cover $\tilde{X}$ of $X$ admits a Stein factorization $\tilde{X}\rightarrow Z_B\rightarrow B$ such that $\deg(Z_B/B)=3$ and this is a contradiction. Hence there are two fibrations $p_B: A_X\rightarrow B$ and $p_C: A_X\rightarrow C$ and $\cM=(p_B^*\cL_B\otimes Q_B)\oplus (p_C^*\cL_C\otimes Q_C)$ and it is easy to see that $\varphi_X$ is   a birational $(\mathbb{Z}/2\mathbb{Z})^2$-cover and can be constructed exactly the same way as in \cite[Example 5.4]{BPS2}. In particular, an \'etale cover of $X$ is birational to a product of two bielliptic curves.
\end{proof}
\subsection{Variations of the eventual paracanonical maps of surfaces}
Let $a: S\rightarrow A$ be a primitive and generating morphism from a smooth projective surface $S$ to an abelian variety. We assume that $a$ is generically finite onto its image and consider the eventual map $\varphi_{K_S, a}$ of $K_S$ with respect to $a$. We have already see that $\varphi_{K_S, a}: S\rightarrow Z_{K_S, a}$ factors birationally through $\varphi_S$.

Note that \cite[Corollary 3.2 and Proposition 3.3]{BPS2} works also  for general $\varphi_{K_S, a}$. Hence $\deg\varphi_{K_S, a}\leq 4$ and $\deg\varphi_{K_S, a}\leq 2$ if $Z_{K_S, a}$ is of general type. Moreover, we also have the decomposition theorem for $a_*\omega_S$, hence the same arguments in Proposition \ref{non-gt} and Theorem \ref{surface} prove actually a variation of Theorem \ref{surface}.
 \begin{theo}\label{variation}
Let $a: S\rightarrow A$ be a generically finite morphism from a smooth surface of general type to an abelian variety. Assume that $a$ is primitive and generating, then the eventual paracanonical map $\varphi_{K_S, a}$ is birational, or is  a birational double cover or is a birational bidouble cover. If $\deg\varphi_{K_S, a}>1$, $Z_{K_S,a}$ can not be of general type. If $\deg\varphi_{K_S, a}=4$, $\varphi_{K_S, a}$ is a birational $(\mathbb{Z}/2\mathbb{Z})^2$-cover and an \'etale cover of $S$ is birational to the product of two bielliptic curves.
 \end{theo}

 \section{Higher dimensions}
 In the surface case, the bounds $\deg\varphi_X\leq 4$ and $\deg\varphi_X\leq 2$ when $\chi(X, \omega_X)=\chi(Z, \omega_Z)$ are deduced from Miyaoka-Yau inequality $9\chi(X, \omega_X)\geq \vol(X)$. These bounds are important for the complete classifications of $\varphi_X$ as we have seen in the proof of Proposition \ref{technical} and Theorem \ref{surface}.
 
  Unfortunately, we do not have such inequalities in higher dimensions. 
  In this section, we will apply the generic vanishing theory to get an effective bound of the eventual paracanonical maps of threefolds. 
  
  \begin{defi} \begin{itemize}
  \item[(1)] Let $X$ be a smooth projective variety. We define $$\tilde{q}(X)=\mathrm{Sup}\{q(\tilde{X})\mid \tilde{X}\xrightarrow{\text{finite abelian \'etale}} X\}.$$  
  \item[(2)]We say that a smooth variety $X$ of maximal Albanese dimension satisfies property $\C1$ if there exists $p_B\in \cC_X$ such that $\dim X_B=1$.
  \end{itemize}
  \end{defi}
\begin{rema}If $X$ is a smooth variety of maximal Albanese dimension  and $\tilde{q}(X)=\dim X$, then for any finite \'etale cover $\tilde{A}_X\rightarrow A_X$, the induced morphism $\tilde{X}\rightarrow \tilde{A}_X$ is the Albanese morphism of $\tilde{X}$.
  
By the decomposition theorem,  property $\C1$ for $X$ is equivalent to one of the following:
\begin{itemize}
\item there exists a direct summand $p_B^*\cF_{B,i}^X\otimes Q_{B,i}$  of $a_{X*}\omega_X$ such that $\dim \Supp\cF_{B,i}^X=1$;
\item For some abelian \'etale cover $\tilde{X}$ of $X$, the corresponding $\tilde{X}_B$ is a smooth curve of genus $\geq 2$.
\end{itemize}
  In particular,  if $X$ satisfies $\C1$, then $\tilde{q}(X)=\infty$.
  \end{rema}
  The proof of Theorem \ref{threefold} is lengthy and can be  roughly divided to the following cases:
  \begin{itemize}
  \item[(1)] The easiest case is that $\tilde{q}(Z)>3$. In this case, we essentially reduce the problem to surface or curve cases. Note that if $\kappa(Z)=1$, then $Z$ satisfies property $\C1$ and  hence $\tilde{q}(Z)>3$.
  \item[(2)] We then assume that $\tilde{q}(X)=3$ and $X$   satisfies property $\C1$. In this case, we will deduce a bound of $\varphi_X$ when $\kappa(Z)=2$ or $3$.
  \item[(3)] We then assume that  $X$ does not satisfy $\C1$ and $\tilde{q}(Z)=3$. Here, the case that $\chi(X, \omega_X)=\chi(Z, \omega_Z)>0$ is particularly interesting. Theorem \ref{=} will be the threefold version of Theorem \ref{technical} and Examples \ref{example1} and \ref{example2} show some new features of $\varphi_X$ in higher dimensions.  We will complete the discussions about the case that $\kappa(Z)=2$ or $3$.
  \item[(4)] Finally we  prove that $\varphi_X\leq 8$ and the equality holds only when $\kappa(Z)=0$ and $\varphi_X$ is a birational $(\mathbb{Z}/2\mathbb{Z})^3$-cover.
  \end{itemize}
  
In this section, we will always take $\pi: \tilde{A}_X\rightarrow A_X$   a sufficiently large \'etale cover so that we are in the situation of Remark \ref{strategy}.

\subsection{Case 1}
Assume that $\tilde{q}(Z)>3$.  In this case, as $\pi: \tilde{A}_X\rightarrow A_X$ is a sufficiently large \'etale cover, $a_{\tilde{Z}}(\tilde{Z})$ is a proper subvariety of $A_{\tilde{Z}}$. Note that $Z$ always satisfies this assumption if $\kappa(Z)=1$.

 We  denote by $V$ the image of $a_{\tilde{Z}}(\tilde{Z})$. Then $\kappa(V)>0$ and by the same argument in Corollary \ref{fiber}, $V$ is fibred by tori and hence $\kappa(V)=1$ or $2$.  
 
 Hence, we consider the commutative diagram:
\begin{eqnarray*}
\xymatrix{\tilde{X}\ar[drr]\ar[r]^{\tilde{\varphi}_X} & \tilde{Z}\ar[dr]\ar@{->>}[r]^t & V\ar[d]^q\ar@{^{(}->}[r] & A_{\tilde{Z}}\ar[d]\\
&& V_B\ar@{^{(}->}[r]  & B,}
\end{eqnarray*}
where $A_{\tilde{Z}}\rightarrow B$ is the quotient such that the induced morphism $V\rightarrow V_B$ is the Iitaka fibration of $V$ and $V_B$ is a subvariety of $B$ of general type.

Recall that $\pi_Z^*\cL\hookrightarrow \tilde{\varphi}_{X*}\omega_{\tilde{X}}$ and $h^0_{a_{\tilde{Z}}}(\pi_Z^*\cL)=\chi(\tilde{Z}, \tilde{\varphi}_{X*}\omega_{\tilde{X}})$.

\begin{lemm}\label{reduction}
In the above  setting,  let $s\in V_B$ be a general point and let $\tilde{X}_s$ and $\tilde{Z}_s$ be the connected fiber of $\tilde{X}$ and $\tilde{Z}$ over $s$, then for $Q\in \Pic^0(A_{\tilde{Z}})$ general, we have $h^0(\tilde{Z}_s, \cL\mid_{\tilde{Z}_s}\otimes Q)=h^0(\tilde{X}_s, \omega_{\tilde{X}_s}\otimes Q)$ and $\deg\varphi_X=\deg\tilde{\varphi}_{X}=\deg(\tilde{X}_s/\tilde{Z}_s)$.
\end{lemm}
\begin{proof}
Apply the decomposition theorem in this case, we write $t_*\tilde{\varphi}_{X*}\omega_{\tilde{X}}=t_*(\omega_{\tilde{Z}}\oplus\tilde{\cQ})=\cF^{\tilde{X}}\oplus \cF$, where $\cF^{\tilde{X}}$ is the M-regular part and $\cF$ is the direct sum of the pull-back of M-regular sheaves on lower dimension abelian varieties.
Since $h^0_{a_{\tilde{Z}}}(\pi_Z^*\cL)=\chi(\tilde{Z}, \tilde{\varphi}_{X*}\omega_{\tilde{X}})$, we know that $t_*\pi_Z^*\cL\hookrightarrow t_*\tilde{\varphi}_{X*}\omega_{\tilde{X}}$ contains $\cF^{\tilde{X}}$. On the other hand, $\cF$ is supported on $V$ and since $V_B$ is of general type, each component of  the support of $\cF$ dominates $B$. Thus there exist $A_{\tilde{Z}}\xrightarrow{p_i} B_i\rightarrow B$ such that $\cF=\oplus_ip_i^*\cF_i$, where $p_i$ are quotient with positive dimensional fibers between abelian varieties and $\cF_i$ are M-regular sheaves on $B_i$.

 Thus   $q_*(\cF\otimes Q)=0$ for $Q\in \Pic^0(A_{\tilde{Z}})$ general. Hence
$q_*(t_*\tilde{\varphi}_{X*}\omega_{\tilde{X}}\otimes Q)=q_*(\cF^{\tilde{X}}\otimes Q)$. Moreover, $\cF^{\tilde{X}}\hookrightarrow t_*\pi_Z^*\cL$ and $\pi_Z^*\cL\hookrightarrow \tilde{\varphi}_{X*}\omega_{\tilde{X}}$, hence
  $q_*(t_*\tilde{\varphi}_{X*}\omega_{\tilde{X}}\otimes Q)=q_*(t_*\pi_Z^*\cL\otimes Q)$.
  Thus we   have $h^0(\tilde{Z}_s, \pi_Z^*\cL\mid_{\tilde{Z}_s}\otimes Q)= h^0(\tilde{X}_s, \omega_{\tilde{X}_s}\otimes Q)$ and $\deg\tilde{\varphi}_{X}=\deg(\tilde{X}_s/\tilde{Z}_s)$.
\end{proof}

\begin{prop}\label{case1}
In Case 1, $\deg \varphi_X\leq 4$.
\end{prop}
\begin{proof}
 Let $K$ be the fiber of $A_{\tilde{Z}}\rightarrow B$. Let $\tilde{X}_s\rightarrow \tilde{Z}_s\rightarrow K_s$ be the natural morphisms. 
 
 If $\dim V_B=1$,
by Lemma \ref{reduction}, $h^0(\tilde{X}_s, \omega_{\tilde{X}_s}\otimes Q)=h^0(\tilde{Z}_s, \pi_Z^*\cL\mid_{\tilde{Z}_s}\otimes Q)$ and $\deg\tilde{\varphi}_{X}=\deg(\tilde{X}_s/\tilde{Z}_s)$. Moreover, by the Severi inequality \cite[Theorem E]{BPS1}, $\vol(\pi_Z^*\cL\mid_{\tilde{Z}_s})\geq 2h^0_{a_{\tilde{Z}}}(\pi_Z^*\cL\mid_{\tilde{Z}_s})=2\chi(\tilde{X}_s, \omega_{\tilde{X}_s})$. On other other hand, $\pi_Z^*\cL\mid_{\tilde{Z}_s}\hookrightarrow (\tilde{\varphi}_{X}\mid_{\tilde{X}_s})_*(\omega_{\tilde{X}_s})$.  Hence $ (\tilde{\varphi}_{X}\mid_{\tilde{X}_s})^*(\pi_Z^*\cL\mid_{\tilde{Z}_s})\hookrightarrow \omega_{\tilde{X}_s}$ and thus $$\vol(\tilde{X}_s)\geq \deg(\tilde{X}_s/\tilde{Z}_s)\vol(\pi_Z^*\cL\mid_{\tilde{Z}_s})\geq 2(\deg\tilde{\varphi}_X)h_{a_X}^0(\omega_{\tilde{X}_s})=2(\deg\tilde{\varphi}_X)\chi(\tilde{X}_s, \omega_{\tilde{X}_s}).$$ By Miyaoka-Yao inequality, $9\chi(\tilde{X}_s, \omega_{\tilde{X}_s})\geq \vol(\tilde{X}_s)$, we conclude that $\deg \varphi_X=\deg\tilde{\varphi}_X\leq 4$.

If $\dim V_B=2$, by  the same argument, we have $\deg \varphi_X\leq 2$.
\end{proof}

In the following cases, we will always assume that  $\dim \tilde{q}(Z)=3$ and hence $a_Z$ and $a_X$ are surjective. In particular, $\kappa(Z)=0$, or $2$, or $3$.

\subsection{Case 2}  

In this subsection, we assume that $\tilde{q}(Z)=3$  and $X$ satisfies $\C1$. Let $p_B\in\cC_X$ such that $\dim X_B=1$. Note that $X_B\simeq B$. Otherwise $g(X_B)\geq 2$ and hence $a_X(X)$ dominates $X_B$ and hence so does $Z$, which contradicts the assumption that $\tilde{q}(Z)=3$. We apply the strategy in Remark \ref{strategy}. Let $\pi: \tilde{A}_X\rightarrow A_X$ be a sufficiently large isogeny. Denote by $G$ the Galois group of $\pi$. Then $G$ is an abelian group and is also the Galois group of $\pi_X$ and $\pi_Z$. Let $C:=\tilde{X}_B$.

 Then from the commutative diagram
\begin{eqnarray*}
\xymatrix{
\tilde{X}\ar[r]\ar[d] &X\ar[d]\\
C\ar[r]^{\epsilon} & B,}
\end{eqnarray*} we see that the induce morphism $C:=\tilde{X}_B\rightarrow B$ is also an abelian cover with Galois group $G_1$, which is naturally a quotient group of $G$.

Let $\hat{Z}$ be a smooth model of the main component of $C\times_B\tilde{Z}$. After birational modifications, we then have $$\tilde{\varphi}_X: \tilde{X}\xrightarrow{\varphi_1} \hat{Z}\xrightarrow{\varphi_2} \tilde{Z},$$ where $\varphi_2$ is  a birational $G_1$-cover.

\begin{lemm}\label{trivial} If $\hat{Z}$ is of general type, then $\deg\varphi_1\leq 2$. The equality holds only if for $c\in C$ general, $\varphi_1\mid_{\tilde{X}_c}: \tilde{X}_c\rightarrow \hat{Z}_c$ is  as in Proposition \ref{technical}, where $\tilde{X}_c$ and $\hat{Z}_c$ are respectively the fibers of $\tilde{X}$ and $Z$ over $c$.
\end{lemm}
\begin{proof}
Note that both  $\tilde{X}$ and $\hat{Z}$ have a fibration to $C$. 
Since $\hat{Z}$ is of general type, by Corollary \ref{intermediate} and the structure of Ein-Lazarsfeld threefold (see  Remark \ref{easyproof}), $\chi(\tilde{X}, \omega_{\tilde{X}})=\chi(\hat{Z}, \omega_{\hat{Z}})>0$.

Let $h_1: \tilde{X}\rightarrow C$ and $h_2: \hat{Z}\rightarrow C$ be the natural fibrations. Then for $Q\in \PA_X$ general, we have $$h^0(C, h_{1*}(\omega_{\tilde{X}}\otimes Q))=\chi(\tilde{X}, \omega_{\tilde{X}})=\chi(\hat{Z}, \omega_{\hat{Z}})=h^0(C, h_{2*}(\omega_{\hat{Z}}\otimes Q)).$$

By Fujita's theorem on  the nefness of $h_{1*}(\omega_{\tilde{X}/C}\otimes Q)$ and $h_{2*}(\omega_{\hat{X}/C}\otimes Q)$ and Riemann-Roch, we conclude that $\rank h_{1*}(\omega_{\tilde{X}/C}\otimes Q)=\rank h_{2*}(\omega_{\hat{X}/C}\otimes Q)$. Hence $\chi(\tilde{X}_c, \omega_{\tilde{X}_c})=\chi(\hat{Z}_c, \omega_{\hat{Z}_c})$ and hence by Proposition \ref{technical}, $\deg \varphi_1\leq 2$.
\end{proof}
\begin{lemm}\label{base}
Under the setting of Lemma \ref{trivial}, for any quotient $p_{B'}\in\cC_X$ such that $p_B\times p_{B'}$ is an isogeny, the  map of $\tilde{X}_{B'}$ induced by the linear system $|K_{\tilde{X}_{B'}}\otimes Q|$ for $Q\in\PB'$, factors through the induced morphism $ \tilde{X}_{B'}\rightarrow \tilde{Z}_{B'}$ and hence the latter is  either birational, or is a  birational $G$-cover, where $G=\mathbb{Z}/2\mathbb{Z}$ or $(\mathbb{Z}/2\mathbb{Z})^2$.
\end{lemm}
\begin{proof}
Since $p_{B'}\in\cC_X$, $\tilde{X}_{B'}$ is of general type.
Under the assumption, $\tilde{X}\rightarrow C\times \tilde{X}_{B'}$ is a generically finite morphism. 
Let $D\in |K_C\otimes P|$ for some general $P\in\PB$. Then  $D+|K_{\tilde{X}_{B'}}\otimes Q|$ is a sub-linear system of $|K_{\tilde{X}}\otimes P\otimes Q|$. 

As $P\otimes Q\in\Pic^0(\tilde{A}_X)$ is general,  $\tilde{\varphi}_X$ is birationally equivalent to the map induced by $|K_{\tilde{X}}\otimes P\otimes Q|$. Hence the map induced by $|K_{\tilde{X}_{B'}}\otimes Q|$  factors birationally through the morphism $\tilde{X}_{B'}\rightarrow \tilde{Z}_{B'}$. Since $\pi$ is large enough and sufficiently divisible, we may assume that the map induced by $|K_{\tilde{X}_{B'}}\otimes Q|$ is the eventual paracanonical map $\varphi_{K_{\tilde{X}_{B'}},a'}$ for the morphism $a':\tilde{X}_{B'}\rightarrow B'$. We have seen in Theorem \ref{surface} and Theorem \ref{variation} that $\varphi_{a'}$ is either birational, or is a  birational $G$-cover with $G=\mathbb{Z}/2\mathbb{Z}$ or $(\mathbb{Z}/2\mathbb{Z})^2$.
\end{proof}

We will need a refined version of Lemma \ref{representative}.
\begin{lemm}\label{refi-repre}  Assume that we have a commutative diagram
\begin{eqnarray*}\xymatrix{
\tilde{X}\ar[r]^{\tilde{\varphi}_X}\ar[d]^{t_X} & \tilde{Z}\ar[d]^{t_Z}\ar[r]^{\tilde{g}}& \tilde{A}_X\ar[d]\\
X_1\ar[r]^{f_1} & Z_1\ar[r]^{g_1} & A_1}\end{eqnarray*} such that $g_1\circ f_1$ is a sub-representative of $\tilde{g}\circ\tilde{\varphi}_X$. Assume that $\chi(\tilde{X}, \omega_{\tilde{X}})=\chi(\tilde{Z}, \omega_{\tilde{Z}})$ and that $f_1$ is a birational $H$-cover. Then there exists at most one irreducible character $\chi$ of $H$ such that $\chi(Z_1, (f_{1*}\omega_{X_1})^{\chi})\neq 0$.

 \end{lemm}
\begin{proof}
By Corollary \ref{intermediate}, we may replace $\tilde{X}$ by a smooth model of the main component of $X_1\times_{Z_1}\tilde{Z}$ and assume that $\tilde{\varphi}_X$ is also a birational $H$-cover.  We may assume that $\pi_Z^*\cL\hookrightarrow (\tilde{\varphi}_{X*}\omega_{\tilde{X}})^{\chi}$. Then the M-regular part of $\tilde{a}_{X*}\omega_{\tilde{X}}$ is contained in $(\tilde{a}_{X*}\omega_{\tilde{X}})^{\chi}$. Because 
$$\tilde{a}_{X*}\omega_{\tilde{X}}=\bigoplus_{\tau\in \mathrm{Char}(H)}(\tilde{a}_{X*}\omega_{\tilde{X}})^{\tau},$$ for any other character $\rho\neq \chi$, $\chi(\tilde{Z}, (\tilde{a}_{X*}\omega_{\tilde{X}})^{\rho})=0$. By the same argument in the proof of Lemma \ref{representative}, $(f_{1*}\omega_{X_1})^{\rho}\hookrightarrow t_{Z*}(\tilde{\varphi}_{X*}\omega_{\tilde{X}})^{\rho}$. Hence for all $\rho\neq \chi$, $\chi(Z_1, (f_{1*}\omega_{X_1})^{\rho})=0$.
\end{proof}

\begin{prop}\label{k=2}
In case 2, if  moreover $\kappa(Z)=2$, then $\deg\varphi_X\leq 4$.
\end{prop}
\begin{proof}
Since $\kappa(Z)=2$, we may assume that  $Z\rightarrow Z_{B_1}$ is the Iitaka fibration of $Z$ for some $p_{B_1}\in \cC_{Z}$ (see for instance \cite[Theorem 13]{Kaw}). Note that, since $\tilde{q}(Z)=3$, $\dim B_1=2$. Moreover, $\tilde{Z}_{B_1}$ is of general type.

There are two cases: either   $p_{B_1}\times p_B: A_X\rightarrow B_1\times B$ is an isogeny or $p_{B_1}$ dominates $p_B$ for any $p_B\in\cC_X$ with $\dim B=1$.

In the first case, the natural morphism $\hat{Z}\rightarrow C\times \tilde{Z}_{B_1}$ is generically finite and surjective and hence $\hat{Z}$ is of general type and we may apply Lemma \ref{trivial}. It suffices to show that $\deg \varphi_2=|G|=2$.   The natural morphisms $C\times \tilde{Z}_{B_1}\xrightarrow{f_1} B\times\tilde{Z}_{B_1}\rightarrow B\times B_1$ is a sub-representative of $\tilde{g}\circ \varphi_2$. Moreover, $f_1$ is a birational $G$-cover and as $C\rightarrow B$ is so and $\tilde{Z}_{B_1}$ is of general type, for any $\chi\in G^*$ non-trivial, $\chi(B\times\tilde{Z}_{B_1}, (f_{1*}\omega_{ C\times \tilde{Z}_{B_1}})^{\chi})\neq 0$. Thus, by Lemma \ref{refi-repre}, $G=\mathbb{Z}/2\mathbb{Z}$.

The second case is more difficult. We may assume that $\tilde{Z}\simeq S\times E$ and $\tilde{A}_X=A_{\tilde{Z}}\simeq A_S\times E$, where $E$ is an elliptic curve and $S$ is a surface of general type. Then $p_{B_1}$ is the natural projection $A_S\times E\rightarrow A_S$.

 We apply Lemma \ref{complementary} to $\tilde{\varphi}_X$ and conclude that there exists $p_{B'}\in\cC_X$ such that $\tilde{A}_X\rightarrow B\times B'$ is an isogeny. Let $S':=\tilde{X}_{B'}$ be the surface of general type. Since $\tilde{Z}\simeq S\times E$ and $\tilde{Z}$ does not have  fibrations to  smooth projective curves of genus $\geq 2$, $\tilde{Z}_{B'}=B'$.
By Lemma \ref{base}, the induced morphism $t: S'\rightarrow B'$ is the eventual paracanonical map of $S'\rightarrow B'$. We claim that $t$ is a birational double cover. Otherwise, by Lemma \ref{base}, this morphism is a birational bidouble cover. But in this case,  we see from Theorem \ref{variation} that $S'$ is birational to a product of curves. This is a contradiction to the assumption that $p_{B_1}$ dominates $p_B$ for any $p_B\in\cC_X$ with  $\dim B=1$.
 
 Note that $C\times S'\rightarrow B\times B'$ is primitive.  We consider  the  natural morphisms
\begin{eqnarray}\label{ms}C\times S'\xrightarrow{f_1} B\times B'\xrightarrow{\mathrm{Identity}}  B\times B'.\end{eqnarray}

 We claim that (\ref{ms}) is a sub-representative of $\tilde{g}\circ\tilde{\varphi}_X$. It suffices to prove that $(C\times S')\times_{B\times B'}\tilde{Z}$ has an irreducible main component.  Note that $\hat{Z}$, which is a smooth model of $C\times_B\tilde{Z}$, is birational to $S_1\times E$, where $S_1$ is a smooth model of $C\times_BS$. In particular, $\hat{Z}_{B'}$ is not of general type. Moreover, $S'=\tilde{X}_{B'}\rightarrow B'$ is a birational double cover. Thus $\hat{Z}_{B'}=B'$ and hence $\hat{Z}\rightarrow  B'$ is a fibration and $S'\times_{B'}\hat{Z}$ has an irreducible main component and we finish the proof of the claim.
 
  By Lemma \ref{refi-repre}, we then   conclude that $G_1=\mathbb{Z}/2\mathbb{Z}$.

Let $X'$ be a smooth model of the main component of $ (C\times S')\times_{B\times B'}\tilde{Z}$. It just suffices to show that the natural map  $\tilde{X}\rightarrow X'$ is birational to finish the proof of the second case.
 We have a commutative diagram
\begin{eqnarray*}
\xymatrix{
\tilde{X}\ar[dr]_{h_1}\ar[r] & X'\ar[d]^{h_2}\ar[r]& S\ar[dl]^{h_3}\\
& C.}\end{eqnarray*}
Moreover, $X'$ is of general type, so $\chi(X', \omega_{X'})>0$ and hence by Corollary \ref{intermediate}, $\chi(\tilde{X}, \omega_{\tilde{X}})=\chi(X', \omega_{X'})$. 
We apply the argument as in Lemma \ref{trivial}. For $c\in C$ general, $\chi(\tilde{X}_c, \omega_{\tilde{X}_c})=\chi(X_c', \omega_{X_c'})$. Then $\tilde{X}_c\rightarrow X_c'$ is either birational or is as in Proposition \ref{technical}. In particular, $q(X_c')=2$. However, $X_c'$ admits a surjective morphism to the fiber $S_c$ of $h_3$, which is a curve of genus $\geq 2$, which is a contradiction.
\end{proof}

\begin{prop}\label{reduction2}
In case 2, if $\kappa(Z)=3$, then $\deg\varphi_X\leq 4$.
\end{prop}
\begin{proof} 

We again have two cases: either $\chi(Z, \omega_Z)>0$ or $\chi(Z, \omega_Z)=0$.

In both cases, we can apply Lemma \ref{trivial} and just need to show that the Galois group $G_1$ of $\varphi_2$ is $\mathbb{Z}/2\mathbb{Z}$. 

 In the first case, we consider $\hat{Z}\xrightarrow{\varphi_2}\tilde{Z}\xrightarrow{\tilde{g}} \tilde{A}_X$. Write $\varphi_{2*}\omega_{\hat{Z}}=\omega_{\tilde{Z}}\oplus \cQ_2$.  By Lemma \ref{complementary} and Corollary \ref{complementary2}, there exists $p_{B_1}\in\cC_{\cQ_2}$ such that $p_B\times p_{B_1}: \tilde{A}_X\rightarrow B\times B_1$ is an isogeny.
 Since $p_{B_1}\in\cC_{\cQ}$, the induced morphism $\hat{Z}_{B_1}\rightarrow \tilde{Z}_{B_1}$ has degree $>1$ and $\chi(\hat{Z}_{B_1}, \omega_{\hat{Z}_{B_1}})>\chi(\tilde{Z}_{B_1}, \omega_{\tilde{Z}_{B_1}})$.  Moreover, as $\varphi_2$ is a birational $G_1$-cover, $\hat{Z}_{B_1}\rightarrow \tilde{Z}_{B_1}$ is also a birational abelian cover whose Galois group $H$, which is a quotient group of $G_1$.

We claim that $G_1=\mathbb{Z}/2\mathbb{Z}$. Since $\chi(\hat{Z}_{B_1}, \omega_{\hat{Z}_{B_1}})>\chi(\tilde{Z}_{B_1}, \omega_{\tilde{Z}_{B_1}})$, for the morphism $\psi:\hat{Z}_{B_1}\rightarrow \tilde{Z}_{B_1}$, there exists a character $\chi\in H^*$ such that the eigensheaf $\cL_{\chi}$ of the sheave $\psi_{*}\omega_{\hat{Z}_{B_1}}$ has $\chi(\tilde{Z}_{B_1}, \cL_{\chi})>0$. Assume that $G_1\neq \mathbb{Z}/2\mathbb{Z}$,   there exists a non-trivial character $\chi'$ of $G_1$ such that $\chi\chi'$ is   a non-trivial character of $G_1$. We then consider the diagonal quotient $C\times \hat{Z}_{B_1}\rightarrow (C\times \hat{Z}_{B_1})/G_1$. Let $\cL_{\chi'}$ be the corresponding character sheaf of $\epsilon_*\omega_C$ on $B$. Then $\chi(C\times\hat{Z}_{B_1}, \omega_{C\times\hat{Z}_{B_1}})-\chi((C\times\hat{Z}_{B_1})/G_1, \omega_{(C\times\hat{Z}_{B_1})/G_1})\geq \chi(B\times \tilde{Z}_{B_1}, \cL_{\chi'}\boxtimes \cL_{\chi})>0$. We note that the morphisms $C\times \hat{Z}_{B_1}\rightarrow (C\times \hat{Z}_{B_1})/G_1\rightarrow B\times B_1$ is a representative of $\tilde{g}\circ \varphi_2$. But $\chi(\hat{Z}, \omega_{\hat{Z}})=\chi(\tilde{Z}, \omega_{\tilde{Z}})$, which is a contradiction to Lemma \ref{representative}.

In the second case, we may assume that $\tilde{Z}$ is an Ein-Lazarsfeld threefold.
By the structure of an Ein-Lazarsfeld threefold, there exists $p_{B'}\in\cC_Z$ such that $p_B\times p_{B'}:\tilde{A}_X\rightarrow B\times B'$ is an isogeny and $\tilde{Z}_{B'}$ is a surface of general type.  The natural morphisms $X_1:=C\times \tilde{Z}_{B'}\xrightarrow{f_1} Z_1:=B\times \tilde{Z}_{B'} \rightarrow B\times B'$ is then a representative of $\tilde{g}\circ\varphi_2$. For any non-trivial character $\chi\in G_1^*$, let $\cL_{\chi}$ be the eigensheaf of $f_{1*}\omega_{X_1}$. Then it is easy to see that $\chi(Z_1, \cL_{\chi})>0$. Then by Lemma \ref{refi-repre} $G_1$ should have only one non-trivial character. Hence $G_1=\mathbb{Z}/2\mathbb{Z}$.
\end{proof}

\subsection{Case 3}

We now assume that $X$ does not satisfy $\C1$ and $\tilde{q}(Z)=3$. In order to simplify notations, we shall replace $\tilde{X}$ (resp. $\tilde{Z}$) by $X$ and $Z$.  Hence no torsion line bundles appear in the decomposition formula of $a_{X*}\omega_X$.

We will first deal with the case $\chi(X, \omega_X)=\chi(Z, \omega_Z)>0$. Then by Lemma \ref{complementary}, $| \cC_{\cQ}|\geq 2$.
Assume that  $p_{B_1}, p_{B_2}\in \cC_{\cQ}\subset\cC_X$. Let  $\PE$ be the neutral  component of $\PB_1\cap\PB_2$.  Let $Y$ be a smooth model of the main component of $X_{B_1}\times_EX_{B_2}$.

\begin{lemm}\label{primitive}
The natural morphism $\gamma: Y\rightarrow B_1\times_EB_2$ is primitive.
\end{lemm}

\begin{proof}
In case 3, for any $p_B\in \cC_X$, $\dim B=2$.
Moreover, for any $p_B\in \cC$, let $q: B\rightarrow E$ be a fibration to an elliptic curve. We consider the commutative diagram:
\begin{eqnarray*}
\xymatrix{
X_B\ar[r]^{t_B}\ar[dr]_h & B\ar[d]^q\\
& E .}
\end{eqnarray*}
Let $F$ be a general fiber of $h$.  We claim that if $t_B^*Q\mid_F$ is trivial for some $Q\in\PB$, then $Q\in\PE$. Otherwise, $L=R^1h_*(\omega_{X_B}\otimes Q)$ is a non-trivial nef line bundle on $E$. If $L$ is ample, then $V^0(L)=\PE$. We then conclude by Koll\'ar's splitting that $Q+\PE$ is a component of $V^2(\omega_X)$. Then by the main theorem of \cite{CJ}, we conclude that $p_E: A_X\rightarrow E$ belongs to $\cC_X$, which is a contradiction to the assumption. If $L$ is numerically trivial, then $h^2(X_B, \omega_{X_B}\otimes Q\otimes L^{-1})\geq h^1(E,R^1h_*(\omega_{X_B}\otimes Q)\otimes L^{-1} )=1$ and hence $Q=L\in\PE$.

Let $Q\in \Pic^0(B_1\times_EB_2)$ such that $\gamma^*Q=\cO_{Y}$. 
We consider the exact sequence of abelian varieties $$0\rightarrow \PE\rightarrow \PB_1\times \PB_2\rightarrow \Pic^0(B_1\times_EB_2)\rightarrow 0,$$ and let $Q_i\in\PB_i$ such that $Q_1\otimes Q_2\mid_{B_1\times_EB_2}\simeq Q$.  

Note that a general fiber $Y_e$ of $Y\rightarrow E$ is birational to $F_1\times F_2$, where $F_i$ is a general fiber of $h_i: X_{B_i}\rightarrow E$. Thus $\gamma^*Q\mid_{F_e}$ is trivial implies that  $t_{B_i}^*Q_i\mid_{F_i}$ is trivial for $i=1,2$. From the above claim, we conclude that $Q_i\in\PE$, for $i=1,2$. Hence $Q$ is trivial.  
\end{proof}

The following theorem is the analogue of Theorem \ref{technical} in dimension $3$ and is probably the most difficult result in this article.  The setting is slightly different.
\begin{theo}\label{=}
Let $X$ be a smooth threefold of maximal Albanese dimension with $\tilde{q}(X)=3$ and $\chi(X, \omega_X)>0$. Assume that the Albanese morphism of $X$ factors as  
 $$a_X: X\xrightarrow{\varphi} Z\xrightarrow{g} A_X ,$$  where $\varphi$ is a morphism of degree $\geq 2$ between smooth projective varieties and  $\chi(X, \omega_X)=\chi(Z, \omega_Z)>0$. Then we have the following possibilities:
 \begin{itemize}
 \item[(A)] $a_X$ is  a  birational $(\mathbb{Z}/3\mathbb{Z})^2$-cover, $\varphi$ is a birational $\mathbb{Z}/3\mathbb{Z}$-cover, and $\varphi_X$ is birationally equivalent to $\varphi$;
 \item[(B)]  $a_X$ is a birational $(\mathbb{Z}/2\mathbb{Z})^2$-cover, $\varphi$ is a birational double cover, and $\varphi_X$ is birationally equivalent to $a_X$;
  \item[(C)] $a_X$ is a birational $(\mathbb{Z}/2\mathbb{Z})^3$-cover and there are several sub-cases:
  \begin{itemize}
   \item[(C1)] $\varphi$ is a birational double cover, $\rank\cF^X=3$, and $\varphi_X$ is a birationally equivalent to $\varphi$;
   \item[(C2)] $\varphi$ is a birational double cover, $\rank\cF^X=2$, and $\varphi_X$, which factors through $\varphi$, is a birational $(\mathbb{Z}/2\mathbb{Z})^2$-cover;
   \item[(C3)] $\varphi$ is a birational double cover, $\rank\cF^X=1$, and $\varphi_X$ is a birationally equivalent to $a_X$;
   \item[(C4)] $\varphi$ is a birational $(\mathbb{Z}/2\mathbb{Z})^2$-cover, $\rank\cF^X=1$, and $\varphi_X$ is a birationally equivalent to $a_X$;
\end{itemize}
 \end{itemize} 
\end{theo}

\begin{proof}   Let $\varphi_*\omega_X=\omega_Z\oplus\cQ$ be the canonical splitting.
 
 For any varieties $Z'$ such that $\mathbb{C}(Z')$ is a intermediate field extension between $\mathbb{C}(X)/\mathbb{C}(Z)$, we can take birational modifications of $X$ and $Z'$ such that $\varphi$ factors  as $X\xrightarrow{\varphi'} Z'\rightarrow Z$. Then we also have $\chi(X, \omega_X)=\chi(Z', \omega_{Z'})$.  If we assume that $\varphi'$ is minimal and prove   Theorem \ref{=} for $\varphi'$, then we are in Cases $\mathrm{(A)}$, $\mathrm{(B)}$, $\mathrm{(C1)}$, $\mathrm{(C2)}$ and $\mathrm{(C3)}$ and, from the structure of $a_X$, the only possible case that $Z'$ is not birational to $Z$ is   $\mathrm{(C3)}$.  In this situation, $\varphi$ should be the Case $\mathrm{(C4)}$. 
 
From the above discussion, we just need to deal the case that $\varphi$ is {\em minimal} and we shall always assume that $\varphi$ is minimal.
 In particular, for any $p_B\in\cC_{\cQ}$, $X$ is birational to the main component of $X_B\times_{Z_B}Z$.
 \\

 {\bf Step 1} {\em 
 We claim that $|\cC_{\cQ}|\geq 3$ or $\deg \varphi=\deg g=2$.}\\

Let $k=\deg \varphi$ and $m=\deg g$.  By the decomposition theorem,
$$a_{X*}\omega_X=g_*\omega_Z  \bigoplus_{p_B\in\cC_{\cQ}}p_B^*\cF_{B}^{\cQ}.$$

We compare the rank of the sheaves on both side and get $km=m+\sum_{p_B\in\cC_{\cQ}}\rank \cF_B^{\cQ}$. Note that for any $p_B\in\cC_{\cQ}$, $\deg(X_B/Z_B)=k$ by the minimality of $\varphi$. 

As  $\dim B=2$, for each $p_B\in\cC_X$, there exit no inclusion relations among these $\PB$. Hence, by (\ref{explanation}), if $p_B\in\cC_{\cQ}$, $a_{B*}\omega_{X_B}=g_{B*}\omega_{Z_B}\oplus \cF^{\cQ}_B$ and hence
 \begin{eqnarray*}&&1+\rank \cF_B^Z+\rank \cF_B^{\cQ}\\&=&\deg(X_B/B)\\&=& k\deg(Z_B/B)=k\times(1+\rank \cF_B^Z).
\end{eqnarray*}
We then conclude that $\rank \cF_B^{\cQ}=(k-1)(1+\rank \cF_B^Z)$ and
\begin{eqnarray}\label{equal}
m=\sum_{p_B\in\cC_{\cQ}}(1+\rank \cF_B^Z).
\end{eqnarray}

By Corollary \ref{complementary2}, we pick $p_{B_1}$ and $p_{B_2}$ two fibrations in $\cC_{\cQ}$ and let $Y$ be as in Lemma \ref{primitive}.
 
Then $Y\times_{B_1\times_EB_2}A_X$ is connected and $a_X$ factors  as $X\rightarrow Y':=Y\times_{B_1\times_EB_2}A_X\rightarrow A_X$. Compare the degree of the morphisms to $A_X$, we get
\begin{eqnarray}\label{inequal}km=\deg a_X\geq \deg(Y'/A_X)= k^2(1+\rank \cF^Z_{B_1})(1+\rank \cF^Z_{B_2}).\end{eqnarray}

Hence $|\cC_{\cQ}|\geq 3$ or $k=m=2$. In particular, $k\geq 3$ implies that $|\cC_{\cQ}|\geq 3$.\\
 
  {\bf Step 2} {\em Construct a representative of $g\circ\varphi$ from any two $p_{B_1}$, $p_{B_2}\in\cC_{\cQ}$.} \\
  
Let $X_B'$ be the normalization of $X_B$ in the Galois closure of the extension of $\mathbb{C}(X_B)/\mathbb{C}(Z_B)$. Let $G_B$ be the Galois group of $\mathbb{C}(X_B')/\mathbb{C}(Z_B)$ and $H_B\subset G_B$ be the Galois group of $\mathbb{C}(X_B')/\mathbb{C}(X_B)$. After birational modifications, we may assume that $G_B$ acts on $X_B'$ and hence $X_B$ is birational to $X_B'/H_B$ and $Z_B$ is birational to $X_B/G_B$.

Let $X'$ be a smooth model of the main component of the fiber product $X\times_{X_B}X_B'$ and we may assume that $G$ acts on $X'$. Then $\mathbb{C}(X')/\mathbb{C}(Z)$ is the Galois closure of $\mathbb{C}(X)/\mathbb{C}(Z)$. Hence all $G_B$ are isomorphic  to a fixed group $G$ and all $H_B$ are  isomorphic to a subgroup $H$ of $G$ for all $p_B\in\cC_{\cQ}$.

Take $p_{B_1}, p_{B_2}\in\cC_{\cQ}$ and let $\PE$ be the neutral component of $\PB_1\cap\PB_2$. Let $X'\rightarrow C\rightarrow E$ be the Stein factorization of the natural morphism $X'\rightarrow E$.
Then we have a natural $G$-equivariant morphism $X'\rightarrow X_{B_1}\times_CX_{B_2}$, where $G$ acts on $X_{B_1}\times_CX_{B_2}$ diagonally. Let $M$ be a smooth model of the main component of $X_{B_1}\times_CX_{B_2}$  and after birational modifications, we may assume that $G$ acts on $M$ and we have a $G$-equivariant surjective morphism $X'\rightarrow M$.

Then let $X_{B_1B_2}$ be the quotient  $M/H$ dominated by $X$ and let $Z_{B_1B_2}$ be the quotient $M/G$. Both varieties have finite quotient singularities and hence are Cohen-Macaulay. Consider the following commutative diagram:
\begin{eqnarray*}
\xymatrix{
X'\ar@/^2pc/[rr]^{/G}_f\ar[r]^{/H}\ar[d] & X\ar[d]\ar[r]^{\varphi} & Z\ar[r]\ar[d]^t & A_X\ar[d]\\
M\ar@/_2pc/[rr]^{/G}_{f'}\ar[r]^{/H}& X_{B_1B_2} \ar[r]^{\varphi'} & Z_{B_1B_2}\ar[r] & B_1\times_EB_2.}
\end{eqnarray*}

Then $X$ is birational to the main component of $X_{B_1B_1}\times_{Z_{B_1B_2}}Z$.

Let $\wZ_{B_1B_2}$ be a connected component of $Z_{B_1B_2}\times_{B_1\times_EB_2}A_X$ and let $\wX_{B_1B_2}$ be the corresponding component of $Z_{B_1B_2}\times_{B_1\times_EB_2}A_X$. We have 
\begin{eqnarray*} 
\xymatrix{
 X\ar[d]\ar[r]^{\varphi} & Z\ar[r]\ar[d]^t & A_X\ar@{=}[d]\\
\wX_{B_1B_2} \ar[r]^{\varphi'} & \wZ_{B_1B_2}\ar[r] & A_X.}
\end{eqnarray*}
The morphisms below is then a representative of $g\circ \varphi$. Moreover, $\tilde{X}_{B_1B_2}$ is of general type, and by Lemma \ref{representative}, $\chi(\wX_{B_1B_2}, \omega_{\wX_{B_1B_2}})=\chi(\wZ_{B_1B_2}, \omega_{\wZ_{B_1B_2}})$.\\

  {\bf Step 3} {\em Let $\hat{X}\xrightarrow{\hat{\phi}} \hat{Z}\xrightarrow{\hat{g}} A_X$ be a minimal representative of $g\circ \varphi$.  We claim that, either after finite \'etale covers,  there exists Galois covers $C_i\rightarrow E_i$ from smooth projective curves to elliptic curves with Galois groups $G_i$, for $i=1,2,3$,  finite groups $\mathcal{G}_1\subset \mathcal{G}_2\subset G_1\times G_2\times G_3$ such that $\phi$ is birational to the quotient $(C_1\times C_2\times C_3)/\mathcal{G}_1\rightarrow (C_1\times C_2\times C_3)/\mathcal{G}_2$, or $\hat{X}\rightarrow A_X$ is a birational $(\mathbb{Z}/2\mathbb{Z})^2$-cover.} \\

We write $\hat{\phi}_*\omega_{\hat{X}}=\omega_{\hat{Z}}\oplus\hat{\cQ}$.  We again have two cases: either $|\cC_{\hat{\cQ}}|\geq 3$ or $|\cC_{\hat{\cQ}}|=2$.
 
  In the first case, there exist $p_{B_1}, p_{B_2}, p_{B_3}\in \cC_{\hat{\cQ}}$.  After an \'etale cover, we may assume that $A=E_1\times E_2\times E_3$ and for $\{i, j, k\}=\{1,2,3\}$, the quotient $p_{B_i}$ is simply the natural quotient $A\rightarrow E_j\times E_k$.
Then, by Step 3 and the assumption that $\hat{g}\circ\hat{\varphi}$ is  a minimal representative,  we know that   $\hat{X}$ is birational to $\hat{X}_{B_iB_j}$ for all $i, j\in\{1, 2, 3\}$, where $\hat{X}_{B_iB_j}$ is a desingularization of the main component of the quotient of the fiber product of $(\hat{X}'_{B_i}\times_{C_k}\hat{X}'_{B_j})/H$ as in the previous step, where $C_k$ is the normalization of $E_k$ in the Galois closure $\hat{X}'$ of $\hat{X}\rightarrow \hat{Z}$. 

We  recall the argument \cite[Theorem 3.4, Step 3]{DJLS} to prove that the fibrations $h_{B_i}: \hat{X}\rightarrow \hat{X}_{B_i}$ are birationally isotrivial. Since $X$ is birational to $(\hat{X}'_{B_1}\times_{C_3}\hat{X}'_{B_2})/H$, the fibers of $h_{B_1}$ over two general points in a general fiber of $\hat{X}_{B_1}\rightarrow E_3$ are birational to each other. By the same reason, the fibers of $h_{B_1}$ over two general points in a general fiber of $\hat{X}_{B_1}\rightarrow E_2$ are birational to each other. Hence two general fibers of $h_{B_1}$ are birational to each other. Then by \cite{BBG}, $h_{B_1}$ is birationally isotrivial and the same arguments works for $h_{B_2} $ and $h_{B_3}$.

Then by \cite[Lemma 3.5]{DJLS}, there exist Galois covers $C_i\rightarrow E_i$ with Galois group $G_i$ and a subgroup $\mathcal{G}_1\subset G_1\times G_2\times G_3$ such that $\hat{X}\rightarrow A$ is birational to the quotient $(C_1\times C_2\times C_3)/\mathcal{G}_1\rightarrow E_1\times E_2\times E_3$ and $\cG_1$ acts on $C_i\times C_j$ faithfully for any $i\neq j$. Since $C_1\times C_2\times C_3\rightarrow E_1\times E_2\times E_3$ is Galois, there exists $\cG_2\subset G_1\times G_2\times G_3$ a subgroup such that $Z$ is birational to $(C_1\times C_2\times C_3)/\cG_2$.

If $|\cC_{\hat{\cQ}}|=2$, we have $\deg\hat{\varphi}=\deg\hat{g}=2$ and it is easy to see that $\hat{g}\circ\hat{\varphi}$ is a birational bidouble cover. \\

  {\bf Step 4} {\em If $|\cC_{\hat{\cQ}}|\geq 3$, then $G_i=\mathbb{Z}/3\mathbb{Z}$, $\cG_1\simeq \mathbb{Z}/3\mathbb{Z}$, $\cG_2\simeq (\mathbb{Z}/3\mathbb{Z})^2$ and $X$ is birational to $\hat{X}$. This is Case $\mathrm{(A)}$.}\\ 
  
Let $g_i=|G_i|$. Note that $\deg(\hat{X}/A)=\frac{g_1g_2g_3}{|\cG_1|}>1$  and $\deg(\hat{Z}/A)=\frac{g_1g_2g_3}{|\cG_2|}$. Moreover, as $\hat{X}$ does not satisfies property $\C1$ and is dominated by $C_1\times C_2\times C_3$, we conclude that $\cC_{\hat{X}}=\cC_{\hat{\cQ}}=\{p_{B_1}, p_{B_2}, p_{B_3}\}$ and $\deg(\hat{X}_{B_i}/\hat{Z}_{B_i})=\deg(\hat{X}/\hat{Z})=\frac{|\cG_2|}{|\cG_1|}$.   Moreover,  $\deg(\hat{X}_{B_i}/B_i)=\frac{g_jg_k}{|\cG_1|}$ and $\deg(\hat{Z}_{B_i}/B_i)\geq\frac{g_jg_k}{|\cG_2|}$\footnote{The action of $\cG_2$ on $C_j\times C_k $ may not be faithful.}. Put everything together, we have 
\begin{eqnarray*}
\frac{g_1g_2g_3}{|\cG_1|}\leq \frac{g_1g_2g_3}{|\cG_2|}+(\frac{1}{|\cG_1|}-\frac{1}{|\cG_2|})(g_1g_2+g_1g_3+g_2g_3).
\end{eqnarray*}
Hence $\frac{1}{g_1}+\frac{1}{g_2}+\frac{1}{g_3}\geq 1$. For   $\{i, j, k\}=\{1, 2, 3\}$, since $X$ does not satisfy  property $\C1$, $\cG_1$ acts faithfully on $C_j\times C_k$, and $p_{B_i}\in\cC_{\hat{\cQ}}$, we have that $\cG_1\twoheadrightarrow G_i$ and hence $g_i$ divides $|\cG_1|$, $|\cG_1|$ divides $g_jg_k$, and  $|\cG_1|<g_jg_k$.
Then, after a calculation of all possible solutions, we see that  $\{g_1, g_2, g_3\}=\{3,3,3\}$ or $\{2,2,2\}$ or $\{2,4,4\}$, other solutions can be ruled out by the conditions on $\cG_1$ and $\cG_2$.

If $\{g_1, g_2, g_3\}= \{2,2,2\}$, we should have $|\cG_2|=4$ and $|\cG_1=2|$. Then we see that  $|\cC_{\hat{\cQ}}|=2$, which contradicts the assumption of this step.

If $\{g_1, g_2, g_3\}= \{2,4,4\}$, we should have $|\cG_1|=4$. Then $\deg(\hat{X}_{B_2}/B_2)=\deg(\hat{X}_{B_3}/B_3)=2$ and $\deg(\hat{X}/A_X)=8$. However, then $\hat{X}$ cannot be birational to $\hat{X}_{B_2B_3}$, which is a contradiction to the minimality of $\hat{g}\circ\hat{\varphi}$.

Hence, we are always in the first case and  then $\cG_1\simeq \mathbb{Z}/3\mathbb{Z}$ and $\cG_{2}\simeq (\mathbb{Z}/3\mathbb{Z})^2$.

We then need to prove that $X$ is birational to $\hat{X}$. 
 It suffices to show that $|\cC_{\cQ}|=3$. If $|\cC_{\cQ}|=3$, by (\ref{equal}) and (\ref{inequal}), $m=3+\rank \cF_{B_1}^Z+\rank \cF_{B_2}^Z+\rank \cF_{B_3}^Z\geq 3(1+\rank \cF_{B_i}^Z)(1+\rank \cF_{B_j}^Z)$ for $i\neq j$. We then see that the only possible case is that $m=3$ and $\cF^Z_{B_i}=0$ for all $i$.  Hence $Z$ is birational to $\hat{Z}$ and $X$ is birational to $\hat{X}$. 
 
  We  argue by contradiction and assume that $|\cC_{\cQ}|\geq 4$. Let $p_{B_4}\in \cC_{\cQ}$ and after permutation of indices, we may assume that the neutral component $\PE_4$ of $\PB_4\cap \PB_3$ is neither $\PE_1$ nor $\PE_2$. Note that $\mathbb{C}(X)/\mathbb{C}(Z)$ is a Galois cover with Galois group $G\simeq \mathbb{Z}/3\mathbb{Z}$. Hence $X_{B_4}\rightarrow Z_{B_4}$ is also birationally a $\mathbb{Z}/3\mathbb{Z}$-cover. Recall that $\hat{X}_{B_3}=(C_1\times C_2)/(\mathbb{Z}/3\mathbb{Z})$. Let $X_4$ be a smooth model of a main component of $(X_{B_4}\times_{E_4}\hat{X}_{B_3})\times_{B_3\times_{E_4}B_4}A_X$. Then $X_4$ admits a $G$-action and let $Z_4=X_4/G$. By Step 3 and the structure of Ein-Lazarsfeld threefolds, we know that $\chi(X_4, \omega_{X_4})=\chi(Z_4, \omega_{Z_4})>0$.  As in Step 3, we consider the commutative diagram
\begin{eqnarray*}
 \xymatrix{
 X_4\ar[d]\ar[r] & Z_4\ar[r]\ar[d] & A_X\ar@{=}[d]\\
\hat{X}_4 \ar[r]& \hat{Z}_4\ar[r] & A_X,}
\end{eqnarray*}
where the morphisms below is a minimal representative of the morphisms above. By the previous steps, we know that $\hat{X}_4\rightarrow A_X$ is a $(\mathbb{Z}/3\mathbb{Z})^2$-cover and for all $p_B\in \cC_{\hat{X}_4}$, the fibration $\hat{X}_4\rightarrow \hat{X}_{4B}$ is birationally isotrivial. We claim that $\PB_4\in \cC_{\hat{X}_4}$. We have a commutative diagram 
\begin{eqnarray*}
\xymatrix{
X_4\ar[d]\ar[r] &\hat{X}_4\ar[d]\ar[r] &A_X\ar[d]\\
X_{B_4}\ar[r] & \hat{X}_{4B_4} \ar[r]& B_4.}
\end{eqnarray*}

Note that a general fiber of $X_4\rightarrow X_{B_4}$ is of degree $3$ over the corresponding fiber of $A\rightarrow B_4$. Hence so is a general fiber of $\hat{X}_4\rightarrow \hat{X}_{4B_4}$. Therefore, $\deg(\hat{X}_{4B_4}/B_4)=3$ and $p_{B_4}\in \cC_{\hat{X}_4}$. Thus $\hat{X}_4\rightarrow \hat{X}_{4B_4}$ is birationally isotrivial and hence so is $X_4\rightarrow X_{B_4}$. Thus $\hat{X}_{B_3}=(C_1\times C_2)/(\mathbb{Z}/3\mathbb{Z})\rightarrow E_4$ is also birationally isotrivial. However, this is a contradiction to \cite[Lemma 5.3]{CDJ}.

Finally, as $\deg(Z/A_X)=3$, $g: Z\rightarrow A_X$ is minimal and hence $\cO_A$ should be the only direct summand of $g_*\omega_Z$, which is not M-regular and hence $\rank \cF^Z=2$. By Lemma \ref{criteria}, $\varphi_Z$ is birational and hence $\varphi_X$ is birationally equivalent to $\varphi$.
 \\

  {\bf Step 5} {\em We now assume that for any minimal representative $\hat{X}\xrightarrow{\hat{\varphi}} \hat{Z}\xrightarrow{\hat{g}} A_X$ of $g\circ\varphi$, $|\cC_{\hat{\cQ}}|=2$ and hence $\hat{g}\circ\hat{\varphi}$ is a $(\mathbb{Z}/2\mathbb{Z})^2$-cover.}\\
   
In this case, $\varphi: X\rightarrow Z$ is a birational double cover and we denote by $\tau$ the birational involution on $X$ of $\varphi$. We shall need some lemmas to deal with this case. First we fix $\hat{X}\rightarrow \hat{Z}\rightarrow A_X$ a minimal representative of $g\circ \varphi$. Let $\cC_{\hat{\cQ}}=\{p_{B_1}, p_{B_2}\}$.  We denote by $r_i=\rank\cF^Z_{B_i}$ for each $p_{B_i}\in\cC_{\cQ}$ and we may assume that $r_1\geq r_2\geq\cdots $. By (\ref{inequal}), we know that $2(1+r_1)(1+r_2)$ divides $m=\sum_{p_{B_i}\in\cC_{\cQ}}(1+r_i)$.

 We have already see in Step 1 that if $|\cC_{\cQ}|=2$, then $\deg\varphi=\deg g=2$ and hence $X$ is birational to $\hat{X}$ and $Z$ is birational to $\hat{Z}$. This is Case $\mathrm{(B)}$. In the following, we shall always assume that $|\cC_{\cQ}|\geq 3$.

\begin{lemm}\label{cQ} If $|\cC_{\cQ}|=3$, then $r_1=1$, $r_2=r_3=0$ and hence $m=\deg(Z/A_X)=4$. If $|\cC_{\cQ}|=4$, then either $r_1=r_2=r_3=r_4=1$ or $r_1=r_2=r_3=r_4=0$.
\end{lemm}
\begin{proof}
Note that $2(1+r_1)(1+r_2)$ divides $3+r_1+r_2+r_3$ implies that $r_1=1$ and $r_2=r_3=0$. Similarly, the condition that $2(1+r_1)(1+r_2)$ divides $4+r_1+r_2+r_3+r_4$ gives us the possible solutions either $r_1=r_2=r_3=r_4=1$ or $r_1=r_2=r_3=r_4=0$.
\end{proof}

Hence we like to control the number $|\cC_{\cQ}|$. We need to study the structure of $(\mathbb{Z}/2\mathbb{Z})^k$-covers to show that $|\cC_{\cQ}|\leq 4$.

\begin{lemm}\label{doubles}
For any $p_B\in \cC_{\cQ}$, there exists a birational double cover $S_B\rightarrow B$ such that $X$ is birational to the main component of $S_B\times_BZ$.
\end{lemm}
\begin{proof}
The proof is very close to the arguments in Step 4. We may assume that $p_B$ is different from $p_{B_1}$ and $p_{B_2}$. Let $X'$ be a smooth model of the main component of the fiber product of $X_B$ with $\hat{X}_{B_1}$ and let $Z'$ be the quotient of the diagonal involution. Take an appropriate \'etale base change, we may assume that $X'\xrightarrow{\varphi'} Z'\xrightarrow{g'} A_X$ is a representative of $g\circ \varphi$.

We then take $\hat{X}'\xrightarrow{\hat{\varphi}'} Z'\xrightarrow{\hat{g}'} A_X$ a minimal representative of $g'\circ\varphi'$. Note that $\hat{g}'\circ\hat{\varphi}'$ is also a minimal representative of $g\circ \varphi$. Hence $\deg(\hat{X}'/A_X)=4$. We then consider the commutative diagram
\begin{eqnarray*}
\xymatrix{
X'\ar[r]\ar[d] & \hat{X}'\ar[r]\ar[d] & A_X\ar[d]\\
X_B\ar[r] & \hat{X}_B\ar[r] & B.}
\end{eqnarray*}
Let $K$ be the fiber of $A_X\rightarrow B$. Then a general fiber $C$ of $X'\rightarrow X_B$ is of degree $2$ over $K$. Hence a general fiber of $\hat{X}'\rightarrow \hat{X}_B$ is also of degree $2$ over $K$ and thus $\hat{X}_B\rightarrow B$ is a birational double cover.
\end{proof}

By Lemma \ref{doubles}, we can construct $(\mathbb{Z}/2\mathbb{Z})^k$-covers of $A_X$, which is dominated by $X$.

Recall the building data of the birational $G$-cover $f: Y\rightarrow A_X$, for $G=(\mathbb{Z}/2\mathbb{Z})^k$. Assume that $$f_*\omega_Y=\bigoplus_{\chi\in G^*}\cL_{\chi},$$ where $\cL_{\chi}$ is the eigensheaf of character $\chi$. We shall denote by $\cL_{\chi}'$ the double dual of $\cL_{\chi}$. Note that if $\cL_{\chi}$ is the pull-back of a rank one sheaf on $B$, then $\cL_{\chi}'$ is the pull-back of an ample line bundle on $B$. Moreover, $a''_*\cO_{X''}=\bigoplus_{\chi\in G^*}\cL_{\chi}'^{-1}$.
By \cite{Pa}, there are effective divisors $D_{\tau}$ on $A_X$, which are part of the branched divisors of $a''$, for each $\tau\in G$, such that  the following condition holds for any $\rho$ and $\sigma$ of $G^*$: 
\begin{eqnarray}\label{data}\cL_{\rho}'+\cL_{\sigma}'=\cL_{\rho\sigma}'+\sum_{\tau \in V(\rho\sigma)}D_{\tau},\end{eqnarray} where $V(\rho\sigma)$ consists of all $\tau\in G$ such that $\rho(\tau)=\sigma(\tau)=-1$.

If $\cL_{\chi}$ is not M-regular, then it is the pull-back on $A_X$ of a M-regular sheaf on a quotient abelian variety $B$ twisted with a torsion line bundle. We shall call $B$ the ample support of $\cL_{\chi}$. Note that by \cite{CJ}, $\PB\subset\PA_X$ is the translate of $V^0(\cL_{\chi})$ through the origin. Note moreover that, if the ample support of $\cL_{\chi}$ is $B$, then so is $\cL_{\chi}'$.
\\

There are still two sub-cases: either for all minimal representative $\hat{g}\circ\hat{\varphi}$ of $g\circ \varphi$, $\chi(\hat{X}, \omega_{\hat{X}})>0$ or there exists one minimal representative such that $\chi(\hat{X}, \omega_{\hat{X}})=0$.

{\bf Subcase 1} 

We shall see that we are in case $\mathrm{(C1)}$ in this situation.\\

In the first case, let $p_{B_3}\in\cC_{\cQ}$. Recall that  the composition of morphisms $\hat{a}: \hat{X}\xrightarrow{\varphi} \hat{Z}\xrightarrow{\hat{g}} A_X$ is a birational $(\mathbb{Z}/2\mathbb{Z})^2$-cover. We write $\hat{a}_*\omega_{\hat{X}}=\cO_{A_X}\oplus\cL_{\chi_1}\oplus\cL_{\chi_2}\oplus\cL_{\chi_1\chi_2}$, where $\chi_1$ and $\chi_2$ is a generator of the character of the Galois group of $\hat{a}$ and $\cL_{\chi_i}$ is the pull-back of a M-regular rank one sheaf on $B_i$. By the assumption that $\chi(\hat{X}, \omega_{\hat{X}})>0$, we know that $\cL_{\chi_1\chi_2}$ is M-regular and $\hat{g}_*\omega_{\hat{Z}}=\cO_{A_X}\oplus\cL_{\chi_1\chi_2}$.

We denote by $X'$ a smooth model of the main component of the fiber product of $\hat{X}\times_{B_3}S_{B_3}$ and let $\varphi': X'\rightarrow Z'$ the diagonal quotient and write the canonical splitting as $\varphi'_*\omega_{X'}=\omega_{Z'}\oplus\cQ'$. Then the induced morphism $a': X'\xrightarrow{\varphi'} Z'\xrightarrow{g'} A_X$ is a birational $(\mathbb{Z}/2\mathbb{Z})^3$-cover. We then denote by $\cL_{\chi_3}$ the eigensheaf of $a'_*\omega_{X'}$ which comes from the double cover $S_{B_3}\rightarrow B_3$. Hence $\cL_{\chi_3}$ is the pull-back of a M-regular sheaf on $B_3$.  By the assumption of this sub-case, $\cL_{\chi_1\chi_3}$ and $\cL_{\chi_2\chi_3}$ are M-regular sheaves. Note that $g'_*\omega_{Z'}=\cO_{A_X}\oplus \cL_{\chi_1\chi_2}\oplus\cL_{\chi_1\chi_3}\oplus\cL_{\chi_2\chi_3}$. Hence $\cL_{\chi_1\chi_2\chi_3}$ is not M-regular and should be a pull-back of a M-regular sheaf on $B_4$. Note that $p_{B_4}$ is different from $p_{B_i}$, for $1\leq i\leq 3$, otherwise we may assume that $p_{B_4}$ is equivalent to $p_{B_1}$, then from the ring structure of $a'_*\cO_{X'}=\mathcal{H}om(a'_*\omega_{X'}, \cO_{A_X})$, we see that  $\cL_{\chi_{2}\chi_3}'$ is a subsheaf of  of $\cL_{\chi_1}'\otimes \cL_{\chi_1\chi_2\chi_3}'$. Then the ample support of $\cL_{\chi_2\chi_3}'$ is dominated by $B_1$ and hence $\cL_{\chi_2\chi_3}'$  and $\cL_{\chi_2\chi_3}$ cannot be M-regular, which is a contradiction.

 Hence $\cC_{\cQ'}=\{p_{B_1}, p_{B_2}, p_{B_3}, p_{B_4}\}$. 

We then claim that $\cC_{\cQ}=\cC_{\cQ'}$. Otherwise, we take $p_{B_5}\in\cC_{\cQ}\setminus\cC_{\cQ'}$. Let $a'': X''\rightarrow A_X$ be the the corresponding birational $G:=(\mathbb{Z}/2\mathbb{Z})^4$-cover by taking the fiber product of $X'$ with $S_{B_5}.$ Similarly, let $\varphi'': X''\rightarrow Z''$ the diagonal quotient and hence the natural morphisms $X''\xrightarrow{\varphi''}Z''\xrightarrow{g''} A_X$ is a representative of $g\circ \varphi$.

We denote by $\cL_{\chi_4}$ the eigensheaf corresponding to the double cover $S_{B_5}\rightarrow B_5$. Then $\chi_1,\ldots, \chi_4$ are generators of the character group $G^*$. We denote by their dual $\tau_1,\ldots. \tau_4$ the generators of $G$. 
Then $a''_*\omega_{X''}=\bigoplus_{\chi\in G^*}\cL_{\chi}$. We say $\chi$ is even, if it is the trivial character, or $\chi_i\chi_j$ for $i\neq j$, or $\chi_1\chi_2\chi_3\chi_4$. We call other characters odd. Then $g''_*\omega_{Z''}=\bigoplus_{\chi\in G^* \; even}\cL_{\chi}$. Since $g''\circ \varphi''$ is a representative of $g\circ \varphi$, $\chi(X'', \omega_{X''})=\chi(Z'', \omega_{Z''})$. Hence for any $\chi$ odd, $\cL_{\chi}$  is not M-regular.

We know that $\cL_{\chi_i}'$ and $\cL_{\chi_i\chi_j\chi_k}'$ are not M-regular for $1\leq i, j, k\leq 4$. We now argue that such birational $G$-cover does not exist.

 The morphism $X''\rightarrow Z''$ is just the quotient by $\tau_1\tau_2\tau_3\tau_4$. Note that by (\ref{data}), $$D_{\tau_1}\leq \cL_{\chi_1}'^{\otimes 2}, \cL_{\chi_1\chi_2\chi_3}'^{\otimes 2}, \cL_{\chi_1\chi_2\chi_4}'^{\otimes 2}, \cL_{\chi_1\chi_3\chi_4}'^{\otimes 2}.$$
If $D_{\tau_1}$ is non-trivial, then it should be the pull-back of an ample divisor on $E$, where $\PE$ is the neutral component of $\PB_1\cap\PB_4$ and the ample supports of $\cL_{\chi_1\chi_2\chi_4}'$ and $\cL_{\chi_1\chi_3\chi_4}'$ also dominate $E$. Let $X'''=X''/\langle \tau_2\tau_3\tau_4\rangle$ and let $Z'''=X''/\langle \tau_1, \tau_2\tau_3\tau_4\rangle$. Then $a''': X'''\xrightarrow{\varphi'''} Z'''\xrightarrow{g'''} A_X$ is also a representative of $g\circ\varphi$. Moreover, $$a'''_*\omega_{X'''}=\cO_{A_X}\oplus\cL_{\chi_2\chi_3}\oplus\cL_{\chi_3\chi_4}\oplus\cL_{\chi_2\chi_4}\oplus \cL_{\chi_1}\oplus \cL_{\chi_1\chi_2\chi_3}\oplus \cL_{\chi_1\chi_2\chi_4}\oplus\cL_{\chi_1\chi_3\chi_4},$$
where the direct sum of the first $4$ terms are $g'''_*\omega_{Z'''}$. Let $c\in E$ be a general point and let $X'''_c\rightarrow Z'''_c\rightarrow A_{Xc}$ be the morphisms of $\varphi'''$ and $g'''$ between the fibers over $c$. Note that $X'''\rightarrow E$ is a fibration, since $X\rightarrow E$ is a fibration. Since the ample supports dominate $E$, $\cL_{\chi_1}, \cL_{\chi_1\chi_2\chi_3}, \cL_{\chi_1\chi_2\chi_4}, \cL_{\chi_1\chi_3\chi_4}$ restricted on $A_{Xc}$ are still not M-regular. Then $\chi(X'''_c, \omega_{X'''_c})=\chi(Z'''_c, \omega_{Z'''_c})$. Note that $X'''\rightarrow A_{Xc}$ is primitive and of degree $8$. Hence we get a contradiction by Theorem \ref{technical}. Thus $D_{\tau_1}$ is trivial. By exactly the same argument, all $D_{\tau_i}$ are trivial for $1\leq i\leq 4$.
  
We then note that $D_{\tau_1\tau_2}\leq \cL_{\chi_1}', \cL_{\chi_2}', \cL_{\chi_1\chi_3\chi_4}', \cL_{\chi_2\chi_3\chi_4}'$. Hence by considering the quotient $X/\langle \tau_3\tau_4\rangle\rightarrow X/\langle\tau_3\tau_4, \tau_1\tau_2\rangle$ and apply the same argument
as before, $D_{\tau_1\tau_2}$ should be trivial and so is $D_{\tau_i\tau_j}$ for all $1\leq i<j\leq 4$.

Moreover, $D_{\tau_1\tau_2\tau_3}, D_{\tau_1\tau_2\tau_3\tau_4}\leq  \cL_{\chi_1}', \cL_{\chi_2}', \cL_{\chi_3}', \cL_{\chi_1\chi_2\chi_3}'$. Hence we see that $D_{\tau}$ is trivial for  all $\tau\in G$. Then $a'''$ should be an \'etale cover, which is a contradiction. Hence $\cC_{\cQ}=\cC_{\cQ'}$ and hence $|\cC_{\cQ}|=4$.

We now claim that $X$ is birational to $X'$ and thus we are in Case $\mathrm{(C1)}$. By Lemma \ref{cQ}, we just need to exclude the case that $r_1=r_2=r_3=r_4=1$. In that case, $\deg(Z_{B_i}/B_i)=2$ and hence $X_{B_i}$, which is birational to the fiber product of $Z_{B_i}$ and $X'_{B_i}$, is a birational $(\mathbb{Z}/2\mathbb{Z})^2$-cover of $B_i$, for $1\leq i\leq 4$. In particular, $\rank\cF^X_{B_i}=3$, for $1\leq i\leq 4$ and $\rank\cF^X=\rank\cF^{X'}=3$. 
Note that $X\rightarrow A_X$ is a birational $(\mathbb{Z}/2\mathbb{Z})^4$-cover. We may assume that $X\rightarrow X'$ is the quotient map for $\tau_4\in (\mathbb{Z}/2\mathbb{Z})^4$  and denote by $\chi_4$ its dual.  We then denote by $\cL_{\chi_4}$ the corresponding eigensheaf for a character $\chi_4$. Then $\cL_{\chi_4}$ is not a direct summand of $\cF^X$ and we may assume that the ample support of $\cL_{\chi_4}$ is $B_1$. Note that $\cL_{\chi_2\chi_4}$ is also not a direct summand of $a'_*\omega_{X'}$ and hence  is also not M-regular. We then consider the quotients $$t: X'':=X/\langle\tau_1, \tau_3\rangle\rightarrow Z'':=X/\langle\tau_1, \tau_3, \tau_4\rangle\rightarrow A_X.$$ Note that $t_*\omega_{X''}=\cO_{A_X}\oplus\cL_{\chi_2}\oplus\cL_{\chi_4}\oplus\cL_{\chi_2\chi_4}$. Because $\cL_{\chi_2}$, $\cL_{\chi_4}$, and $\cL_{\chi_2\chi_4}$ are not M-regular, $X''$ has vanishing holomorphic Euler characteristic and $X''$ is of general type, since the ample supports of $\cL_{\chi_2}$ and $\cL_{\chi_4}$ are different. Thus, $X''\rightarrow Z''\rightarrow A_X$ is a minimal representative of $g\circ \varphi$, which is a contradiction to the assumption of subcase 1.\\

{\bf Subcase 2}

Cases $\mathrm{(C2)}$ and $\mathrm{(C3)}$ will occur.\\

We then deal with the case that one representative $\hat{X}$ has vanishing holomorphic Euler characteristic.  We may assume that $\hat{X}$ is birational to an Ein-Lazarsfeld theefold, namely there exits double covers $C_i\rightarrow E_i$ to elliptic curves with the involution $\tau_i$, for $1\leq i\leq 3$, such that $\hat{X}$ is birational to the diagonal quotient $(C_1\times C_2\times C_3)/\langle\tau_1\times\tau_2\times\tau_3\rangle$ and hence $A_X=E_1\times E_2\times E_3$. We then write $$\hat{a}_*\omega_{\hat{X}}=\cO_{A_X}\oplus \cL_{\chi_1}\oplus\cL_{\chi_2}\oplus\cL_{\chi_1\chi_2},$$ where $\chi_1$ and $\chi_2$ are generators of the dual of the Galois group of $\hat{a}$ and $\cL_{\chi}$ are the corresponding eigensheaves, which are line bundles.
We may assume that the ample support of $\cL_{\chi_1}$ is $B_1:=E_2\times E_3$, the ample support of $\cL_{\chi_2}$ is $B_2:=E_1\times E_3$, the ample support of $\cL_{\chi_1\chi_2}$ is $B_3:=E_1\times E_2$, and $\cL_{\chi_1\chi_2}$ is a direct summand of $\hat{g}_*\omega_{\hat{Z}}$. In particular, $\cC_{\hat{\cQ}}=\{p_{B_1}, p_{B_2}\}$.
 
Denote by $D_i'$ the branched divisors of $C_i\rightarrow E_i$. Let $D_i$ be the pull-back of $D_i'$ on $A_X$. Then $\cL_{\chi_i}^{\otimes 2}=D_j+D_k$, for all $\{i, j, k\}=\{1, 2, 3\}$.
 
If $|\cC_{\cQ}|=3 $, by Lemma \ref{cQ}, we know that  $\deg a_X=8$. Let $\{p_B, p_{B_1}, p_{B_2}\}=\cC_{\cQ}$. Let $S_B\rightarrow B$ be the birational double cover as in Lemma \ref{doubles} and denote by $\cL_{\chi_3}$ the corresponding eigensheaf. If $p_{B}$ is different from $p_{B_3}$,  from the structure of Ein-Lazarsfeld variety,  a smooth model of the fiber product of $S_B$ with $\hat{X}_{B_3}=(C_1\times C_2)/\langle \tau_1\times\tau_2\rangle$ has positive holomorphic Euler characteristic. Hence  $\cL_{\chi_1\chi_2\chi_3}$ is M-regular. However it is impossible, since $\cL_{\chi_1\chi_2\chi_3}$ is not a direct summand of $g_*\omega_Z$. Hence $p_B=p_{B_3}$.
 Thus   $X$ is birational to the main component of the fiber product of $S_{B}\times_{B}\hat{X}$  and $a_X: X\rightarrow A_X$ is a birational $(\mathbb{Z}/2\mathbb{Z})^3$-cover.   We have $g_*\omega_Z=\cO_{A_X}\oplus\cL_{\chi_1\chi_2}\oplus \cL_{\chi_1\chi_3}\oplus\cL_{\chi_2\chi_3}$.
Depending on the structure of $S_{B}\rightarrow B$, either one of $\cL_{\chi_1\chi_3}$and $\cL_{\chi_2\chi_3}$ is M-regular  or both are M-regular. In the first case, $\varphi_X$ is birational equivalent to $a_X$ and we are in Case $\mathrm{(C3)}$. In the latter case, $\varphi_X$ is birational equivalent to the morphism $X\rightarrow (C_1\times C_2)/\langle \tau_1\times\tau_2\rangle\times E_3$ and we are in Case  $\mathrm{(C2)}$.

If $|\cC_{\cQ}|\geq 4$, we pick $p_{B_4}\in\cC_{\cQ}$, which is different from $p_{B_i}$ for $1\leq i\leq 3$. By Lemma \ref{doubles}, we may take the rank $1$ eigensheaf $\cL_{\chi_3}$, whose ample support is $B_4$.

Let $X'$ be a smooth model of the main component of the fiber product of $\hat{X}$ with $S_{B_4}$ and denote by $\tau$  the  involution such that the morphisms $ a': X'\xrightarrow{\varphi'} Z':=X'/\langle \tau\rangle\xrightarrow{g'} A_X$ is a representative of $g\circ \varphi$. We always denote by $\varphi'_*\omega_{X'}=\omega_{Z'}\oplus \cQ'$ the canonical splitting.

Then $\cL_{\chi_1\chi_2}$ is $\tau$-invariant and $\cL_{\chi_3}$ is $\tau$-anti-invariant, hence  $\cL_{\chi_1\chi_2\chi_3}$ is  also  $\tau$-anti-invariant and thus is not M-regular.  Moreover, $\cO_{A_X}\oplus\cL_{\chi_1\chi_2}\oplus\cL_{\chi_3}\oplus\cL_{\chi_1\chi_2\chi_3}$ is the pushforward of another variety with vanishing holomorphic Euler characteristic and of general type. Thus, from the structure of Ein-Lazarsfeld threefolds, we conclude that we may assume that $\PB_4$ contains either $\PE_1$ or $\PE_2$. We may assume that $\PB_4$ contains $\PE_1$. In this case, by the structure of Ein-Lazarsfeld threefolds, the branched divisor of $S_{B_4}\rightarrow B_4$ is the sum of $D_1$ with an effective divisor, which is the pull-back on $B_4$ of an ample divisor on an elliptic curve $E_4$. Since $p_{B_4} $ is different from $p_{B_2}$ and $p_{B_3}$, and $\PE_1\subset \PB_4$, we know that $\PE_4\nsubseteq \PB_2$ and $\PE_4\nsubseteq \PB_3$  and in particular, $p_{E_4}$ is  different from $p_{E_1}$, $p_{E_2}$, and $p_{E_3}$.

Let $X''$ be a smooth model of the main component of the fiber product $\hat{X}_{B_2}\times_{E_1}S_{B_4}$. Then from the structure of $\hat{X}_{B_2}$ and $S_{B_4}$, $\chi(X'', \omega_{X''})=0$. In particular, $\cL_{\chi_2\chi_3}$ is not M-regular, whose ample support $B_5$ dominates $E_3$ and $E_4$. Note that $\cL_{\chi_2\chi_3}$ is a direct summand of $g_*\omega_Z$, because both $\cL_{\chi_2}$ and $\cL_{\chi_3}$ are $\tau$-anti-invariant. Similarly, we see that $\cL_{\chi_1\chi_2\chi_3}$, which is $\tau$-anti-invariant, is not M-regular, whose ample support $B_6$ dominates $E_2$ and $E_4$. Thus $g'_*\cQ'=\cL_{\chi_1}\oplus\cL_{\chi_2}\oplus\cL_{\chi_3}\oplus\cL_{\chi_1\chi_2\chi_3}$. Note that $p_{B_6}$ is different from $p_{B_4}$, otherwise $\PE_1+\PE_4=\PE_2+\PE_4$ and we conclude that $\PE_4\subset \PE_1+\PE_2=\PB_3$, which is a contradiction.

In summary, $\cC_{\cQ'}=\{p_{B_1}, p_{B_2}, p_{B_4}, p_{B_6}\}$ and $g'_*\omega_{Z'}=\cO_{A_X}\oplus \cL_{\chi_1\chi_2}\oplus\cL_{\chi_2\chi_3}\oplus\cL_{\chi_1\chi_3}$. Since $\chi(X', \omega_{X'})=\chi(Z', \omega_{Z'})>0$, $\cL_{\chi_1\chi_3}$ is M-regular.

We then claim that $\cC_{\cQ}=\cC_{\cQ'}$.  Otherwise, let $p_{B}\in \cC_{\cQ}\setminus \cC_{\cQ'}$ and let $\cL_{\chi}$ be the corresponding eigensheaf. Then as before, $\cL_{\chi\chi_1\chi_2}$ is not M-regular. Thus $B$ dominates either $E_1$ or $E_2$. Moreover, $\cL_{\chi\chi_2\chi_3}$ is also not M-regular, thus $B$ dominates either $E_3$ or $E_4$.   Then $p_B\in \cC_{\cQ'}$, which is a contradiction.

We then show that $X$ is birational to $X'$ and hence we are in Case $\mathrm{(C3)}$. By Lemma \ref{cQ}, we just need to exclude the case that $r_1=r_2=r_3=r_4=1$. In that case, $\rho: Z\rightarrow Z'$ is a birational double cover and $$g_*\omega_Z=g'_*\omega_{Z'}\oplus \cF^Z_{B_1}\oplus \cF^Z_{B_2}\oplus \cF^Z_{B_4}\oplus \cF^Z_{B_6}.$$
We may assume that $\cF^Z_{B_6}=\cL_{\chi_4}$. We have seen that $B_6$ does not dominate $E_1$, because $p_{B_6}$  is different from $p_{B_4}$. Moreover, $B_6$ does not dominate $E_3$, because $p_{B_i}$ is different from $p_{B_1}$. Then by the structure of Ein-Lazarsfeld variety, we know that $\cL_{\chi_2\chi_4}$ is M-regular, but is not a direct summand of $g_*\omega_Z$, which is a contradiction.
\end{proof}

\begin{prop} Assume that any finite abelian cover of $X$ does not have a fibration to smooth projective curves of genus $\geq 2$ and the eventual paracanonical image $Z$ is of general type with $\chi(Z, \omega_Z)=0$. Then $\deg\varphi_X\leq 5$
\end{prop}
\begin{proof}
After \'etale cover of $A_X$, we may assume that $Z$ is birational to an Ein-Lazarsfeld threefold. We first claim that $\mathbb{C}(X)/\mathbb{C}(Z)$ is a minimal field extension. Assume that $\varphi$ factors as $X\xrightarrow{\phi} Y\xrightarrow{\psi} Z$ with $\deg\psi>1$. Then by \cite{CDJ}, $\chi(Y, \omega_Y)>0$. Hence by Lemma \ref{intermediate}, $\chi(X, \omega_X)=\chi(Y, \omega_Y)>0$. However, $\deg(Y/A_X)=4\deg\psi>4$ and by Theorem \ref{=}, $\phi$ is a birational morphism.

Write $\varphi_{X*}\omega_X=\omega_Z\oplus\cQ$.
Then for any $p_B\in \cC_{\cQ}$, $X $ is birational to the main component of $X_B\times_{Z_B}Z$. We may assume that $\deg\varphi_X=k$. As in Step 2 of the proof of Theorem \ref{=}, we then have 
$$4k=4+\sum_{p_B\in\cC_{\cQ}}\rank\cF^{\cQ}_{B}+\rank \cF^X.$$
If $B\in \cC_Z$, then $\deg(X_B/B)=k\deg(Z_B/B)=2k$ and hence $\rank\cF^{\cQ}_{B}=2k-2$, otherwise $\rank\cF^{\cQ}_{B}=k-1$. Thus $(k-1)$ divides $\rank\cF_X$. On the other hand, $\cF^X=\cF_Z$ is contained in $g_*\cL$. Thus $\rank\cF_X\leq 4$ and we have $k\leq 5$.
\end{proof}

We then finish the discussion when $\kappa(Z)=3$.
\begin{theo}\label{k=2-final}Let $X$ be  a threefold of maximal Albanese dimension with $\chi(X, \omega_X)>0$ and let $Z$ be the image of the eventual paracanonical map $\varphi_X$ of $X$. If $\kappa(Z)=3$, then $\deg\varphi_X\leq 5$.
\end{theo}

 \begin{exam}\label{example1}
There exists a birational $\mathbb{Z}/3\mathbb{Z}$-cover  $\varphi: X\rightarrow Z$  between $3$-folds such that $a_X: X\rightarrow  A_X$ is  a birational $(\mathbb{Z}/3\mathbb{Z})^2$-cover, which factors through $\varphi$, and $\chi(X, \omega_X)=\chi(Z, \omega_Z)>0$.

Let $\rho_i: C_i\rightarrow E_i$ be $\mathbb{Z}/3\mathbb{Z}$-covers of elliptic curves for $i=1$, $2$, $3$ and $g(C_i)>0$. Fix $\tau$ a generator of $\mathbb{Z}/3\mathbb{Z}$ and let $\chi\in (\mathbb{Z}/p\mathbb{Z})^{\vee}$ be a generator. Then we write
\begin{eqnarray*}
\rho_{i*}\omega_{C_i}=\bigoplus_{k=0}^{2}L_{i,\chi^k},
\end{eqnarray*}
where $L_{i, \chi^k}$ is the $\chi^k$-eigensheaf, $L_{i, \chi^0}=\cO_{E_i}$ and other eigensheaves are ample line bundle.  Then $Y=C_1\times C_2\times C_3\rightarrow E_1\times E_2\times E_3$ is a $(\mathbb{Z}/3\mathbb{Z})^3$-cover. Let $H$ be the cyclic subgroup of $(\mathbb{Z}/3\mathbb{Z})^3$ generated by $(\tau,  \tau, \tau)$ and let $G$ be the subgroup generated by $H$ and $(\tau, 2\tau, 0)$. Denote by $X'$ the quotient  $Y/H$ and $Z'$ the quotient $Y/G$. We then have the commutative diagram:

\begin{eqnarray*}
\xymatrix{ Y\ar[r] &X'\ar@/^2pc/[rr]^{f'}\ar[r]^{\varphi'} &Z'\ar[r]^(.2){g'} &A:=E_1\times E_2\times E_3.}
\end{eqnarray*}
Note that 
\begin{eqnarray*}f'_*\omega_{X'}=(a_{Y*}\omega_Y)^H=\bigoplus_{\substack{(a,b,c)\in (\mathbb{Z}/3\mathbb{Z})^3\\a+b+c=0}} L_{1, \chi^a}\boxtimes L_{2, \chi^b}\boxtimes L_{3, \chi^c},
\end{eqnarray*}
and \begin{eqnarray*}g'_*\omega_{Z'}=(a_{Y*}\omega_Y)^G&=&\bigoplus_{\substack{(a,b,c)\in (\mathbb{Z}/3\mathbb{Z})^3\\a+b+c=0\\a+2b=0}} L_{1, \chi^a}\boxtimes L_{2, \chi^b}\boxtimes L_{3, \chi^c}\\
&=&\cO_A\oplus (L_{1,\chi}\boxtimes L_{2, \chi}\boxtimes L_{3, \chi}) \oplus (L_{1,\chi^2}\boxtimes L_{2, \chi^2}\boxtimes L_{3, \chi^2}).
\end{eqnarray*}

Then we see that $\chi(A, f'_*\omega_{X'})=\chi(A, g'_*\omega_{Z'})>0$. Let $X$ and $Z$ be respectively smooth models of $X'$ and $Z'$ and after modifications we may assume that we have the commutative diagram
\begin{eqnarray*}
\xymatrix{
X\ar[r]^{\varphi} \ar[d]^{\rho_X}& Z\ar[dr]\ar[d]^{\rho_Z}\\
X'\ar[r] & Z'\ar[r] & A.}
\end{eqnarray*}
Since $X'$ and $Z'$ have quotient singularities, $\rho_{X*}\omega_X=\omega_{X'}$ and $\rho_{Z*}\omega_X=\omega_{Z'}$. Hence $\chi(X, \omega_X)=\chi(Z, \omega_Z)>0$.  Moreover, it is also clear that the natural morphism $X\rightarrow A$ and $Z\rightarrow A$ are respectively the Albanese morphisms of $X$ and $Z$.

 \end{exam}

\begin{exam}\label{example2}
There exists a birational $\mathbb{Z}/2\mathbb{Z}$-cover  $\varphi: X\rightarrow Z$  between $3$-folds such that $a_X: X\rightarrow  A_X$ is  a birational $(\mathbb{Z}/2\mathbb{Z})^3$-cover, which factors through $\varphi$, and $\chi(X, \omega_X)=\chi(Z, \omega_Z)>0$. 

 Let $\rho_i: C_i\rightarrow E_i$ be double covers from smooth curves of genus $\geq 2$ to elliptic curves with bielliptic involutions $\sigma_i$, for $i=1,2$. Then $C_1\times C_2\rightarrow E_1\times E_2 $ is a $(\mathbb{Z}/2\mathbb{Z})^2$-cover and we take $S\rightarrow C_1\times C_2$ be a double cover such that the induced morphism $g: S\rightarrow E_1\times E_2$ is a  $G:=(\mathbb{Z}/2\mathbb{Z})^3$-cover. Then let $\tau_1$ (resp. $\tau_2$) be an involution of $S$, which is a lift of $\sigma_1\times\mathrm{id}_{C_2}$ (resp. $\mathrm{id}_{C_1}\times \sigma_2$) on $C_1\times C_2$. Let $\tau_3$ be an involution on $S$ such that $\tau_1$, $\tau_2$, and $\tau_3$ generate the Galois group $G$ of $g$. We denote by $\tau_i^*$ the dual of $\tau_i$ on the character group $G^*$. Then $$g_*\omega_S=\cO_{E_1\times E_2}\bigoplus \big(\cL_{\tau_1^*}\oplus \cL_{\tau_2^*}\oplus \cL_{\tau_3^*}\big)\bigoplus\big(\cL_{\tau_1^*\tau_2^*}\oplus \cL_{\tau_1^*\tau_3^*}\oplus \cL_{\tau_2^*\tau_3^*}\big)\bigoplus \cL_{\tau_1^*\tau_2^*\tau_3^*},$$ where we denote by $\cL_{\chi}$ the $\chi$-eigensheaf of $g_*\omega_X$ for $\chi\in G^*$. In particular, $\cL_{\tau_i^*}$ is the pull-back of an ample line bundle on $E_i$ for $i=1, 2$. 
 
 We claim that we can choose $S$  general  so that for any other non-trivial character $\chi$, $\cL_{\chi}$ is M-regular on $E_1\times E_2$. We can take $S'\rightarrow E_1\times E_2$ to be a double cover such that the branched locus  is a  smooth irreducible ample divisor, which intersect transversally with the branched divisors of $C_1\times C_2\rightarrow E_1\times E_2$. Let $S=(C_1\times C_2)\times_{E_1\times E_2}S'$. Then $S\rightarrow C_1\times C_2 $ is a double cover and it is easy to see that $S\rightarrow E_1\times E_2$ is a $G$-cover with the desired property.

Let $\rho: C\rightarrow E$ be another double cover from a smooth curve of genus $\geq 2$ to an elliptic curve with the involution $\sigma$. We write $\rho_{*}\omega_{C}=\cO_{E_3}\oplus \cL_{\sigma^*}$.

Let $X$ be a smooth model of $(S\times C)/\langle \tau_1\tau_2\times \sigma \rangle$ and let $Z$ be a smooth model of $(S\times C)/\langle \tau_1\tau_2\times \sigma, \tau_3\times\sigma \rangle$ and after further birational modifications, we have the natural morphisms $a: X\xrightarrow{\varphi} Z\xrightarrow{g} A:=E_1\times E_2\times E_3$. Then we have $$a_*\omega_X=\cO_A\oplus \cL_{\tau_1^*\tau_2^*}\oplus\cL_{\tau_3^*}\oplus\cL_{\tau_1^*\tau_2^*\tau_3^*}\oplus (\cL_{\tau_1^*}\boxtimes\cL_{\sigma^*})\oplus (\cL_{\tau_2^*}\boxtimes\cL_{\sigma^*})\oplus (\cL_{\tau_1^*\tau_3^*}\boxtimes\cL_{\sigma^*})\oplus (\cL_{\tau_2^*\tau_3^*}\boxtimes\cL_{\sigma^*}),$$
and $$g_*\omega_Z=\cO_A\oplus  \cL_{\tau_1^*\tau_2^*}\oplus(\cL_{\tau_1^*\tau_3^*}\boxtimes\cL_{\sigma^*})\oplus (\cL_{\tau_2^*\tau_3^*}\boxtimes\cL_{\sigma^*}).$$
Hence $(\cL_{\tau_1^*\tau_3^*}\boxtimes\cL_{\sigma^*})\oplus (\cL_{\tau_2^*\tau_3^*}\boxtimes\cL_{\sigma^*})$ is the M-regular part of both $a_*\omega_X$ and $g_*\omega_Z$ and thus $\chi(X, \omega_X)=\chi(Z, \omega_Z)$.
\end{exam}

The following proposition finishes the discussion when $\kappa(Z)=2$.
\begin{prop}\label{k=2-2}
In case 3, if $\kappa(Z)=2$, then $\deg\varphi_X\leq 6$.  
\end{prop}
\begin{proof}
We take a sufficiently large \'etale cover $\tilde{A}_X\rightarrow A_X$ and suppose the corresponding $\tilde{Z}$ is birational to $S\times E$. We denote $B:=A_S$ and $p_B:\tilde{A}_X=A_{\tilde{Z}}\rightarrow B$ the natural projection.

 If $\varphi_X$ is a birational double cover, then nothing needs to be proved. Otherwise, $\rank \cQ\geq 2$ and hence $\cC_{\cQ}\neq\emptyset$. By Lemma \ref{complementary}, there exists $p_{B_1}\in\cC_{\cQ}$, which is different from $p_B$. Let $\hat{Z}$ be a smooth model of the main component of the fiber product $\tilde{X}_{B_1}\times_{B_1}\tilde{Z}$. Then $\hat{Z}$ is of general type.
 
 Either $\chi(\hat{Z}, \omega_{\hat{Z}})>0$ and hence $\chi(\tilde{X}, \omega_{\tilde{X}})=\chi(\hat{Z}, \omega_{\hat{Z}})$ or we may assume that $\hat{Z}$ is birational to an Ein-Lazarsfeld threefold.
 
 In the first case, if $\tilde{X}\rightarrow \hat{Z}$ is not birational, we apply Theorem \ref{=}. It can happen only in Case $\mathrm{(C2)}$. Then $\varphi_X$ is a birational bidouble cover. 
 
 We then assume that $\tilde{X}$ is birational to the main component of $\tilde{X}_{B_i}\times_{B_i}\tilde{Z}$ for any $p_{B_i}\in\cC_{\cQ}$. Since $\tilde{X}$ does not admit any fibration to a smooth projective curve of genus $\geq 2$, we see that the natural morphism $\tilde{X}\rightarrow S$ is a fibration. Hence $p_B\notin \cC_{\cQ}$. Thus  for any $p_{B_i}\in \cC_{\cQ}$, $\deg(X_{B_i}/B_i)=\deg(\varphi_X):=k$.
 By the decomposition theorem, we then have $$\deg(\tilde{X}/\tilde{A}_X)=k\deg(\tilde{Z}/\tilde{A}_X)=\deg(\tilde{Z}/\tilde{A}_X)+(k-1)|\cC_{\cQ}|+\rank\cF^{\tilde{X}}.$$ If $|\cC_{\cQ}|=1$, we have
  $(k-1)(\deg(\tilde{Z}/\tilde{A}_X)-1)\leq \rank\cF^{\tilde{X}}\leq \deg(\tilde{Z}/\tilde{A}_X).$ Hence $k\leq 3$.  if $|\cC_{\cQ}|\geq 2$, pick $p_{B_1} \in\cC_{\cQ}$, $\tilde{X}$ is birational to the main component of $\tilde{X}_{B_1}\times_{B_1}\tilde{Z}$. Let $\PE_1$ be the neutral component of $\PB_1\cap\PB$, then $\tilde{X}$ is birational to an \'etale cover of the main component of $S\times_{E_1}\tilde{X}_{B_1}$.
 Then we see that  the only fibrations belong to $\cC_X$ which dominates $E_1$ are $p_B$ and $p_{B_1}$. Hence pick $p_{B_2}\in\cC_{\cQ}$, the neutral component $\PE_2$ of $\PB_2\cap\PB$ is different from $\PE_1$ and $\tilde{X}$ is birational to an \'etale cover of the main component of $S\times_{E_2}\tilde{X}_{B_2}$. By the same argument as in Step 3 of the proof of Theorem \ref{=}, we conclude that the fibration $\tilde{X}\rightarrow S$ is birational isotrivial. Hence there exist a birational $G$-cover $S_1\rightarrow S$ and a smooth projective curve $C$ with a faithful $G$-action such that $\tilde{X}$ is birational to the diagonal quotient $(S_1\times C)/G$. Moreover, since $q(\tilde{X})=3$, we may assume that $C/G\simeq E$. Moreover, $\tilde{X}_{B_i}\rightarrow E_i$ are birational isotrivial fibrations with a general fiber $C$ for $i=1,2$. Hence there exists smooth projective curves $C_1$ and $C_2$ both with faithful $G$-actions such that $\tilde{X}_{B_i}$ is birational to the diagonal quotient $(C_i\times C)/G$. Hence $X_1:=(C_1\times C_2\times C)/G\xrightarrow{t}Z_1:= (C_1\times C_2)/G\times E\rightarrow  A_{(C_1\times C_2)/G}\times E $ is a representative of $g\circ \varphi$. By Lemma \ref{refi-repre}, there exists at most one character $\chi$ of $G$ such that $\chi(Z_1, (t_*\omega_{X_1})^{\chi})\neq 0$. It is easy to see that it is only possible that $G=\mathbb{Z}/2\mathbb{Z}$. Hence $\deg\varphi_X=|G|=2$.
 
In the second case, we may denote  $\cC_{\hat{Z}}=\{p_{B_1}, p_{B_2}, p_{B}\}$.
Note that $t: \tilde{X}\rightarrow \hat{Z}$ is minimal, by Theorem \ref{=} and Corollary \ref{intermediate}. Denote $t_*\omega_{\tilde{X}}=\omega_{\hat{Z}}\oplus\hat{\cQ}$. Then for any $p_B\in\cC_{\hat{\cQ}}$, $\tilde{X}$ is birational  to the main component of $\tilde{X}_B\times_{\hat{Z}_B}\hat{Z}$. Denote by $k=\deg t$. Then by the decomposition theorem, we have \begin{eqnarray*}\deg(\tilde{X}/\tilde{A}_X)=4k &=&\deg(\hat{Z}/\tilde{A}_X)+\sum_{p_B\in\cC_{\hat{\cQ}}}(\deg(\tilde{X}_B/B)-\deg(\hat{Z}_B/B))+\rank\cF^{\tilde{X}}\\&=&4+\sum_{p_B\in\cC_{\hat{\cQ}}}(k-1)\deg(\hat{Z}_B/B)+\rank\cF^{\tilde{X}}.\end{eqnarray*} Note that in this case $\deg(Z/\tilde{A}_X)=2$. Hence $\rank\cF^{\tilde{X}}\leq 2$. Thus $k-1\leq\rank\cF^{\tilde{X}}\leq 2 $. We then have $\deg\varphi_X\leq 6$. 
\end{proof}
With both Proposition \ref{k=2} and \ref{k=2-2}, we have
\begin{theo}\label{k=2-final}Let $X$ be  a threefold of maximal Albanese dimension with $\chi(X, \omega_X)>0$ and let $Z$ be the image of the eventual paracanonical map $\varphi_X$ of $X$. If $\kappa(Z)=2$, then $\deg\varphi_X\leq 6$.
\end{theo}
\subsection{Case 4}

In this case, we can simply assume that $Z=A_X$ and $\varphi_X=a_X$ is the Albanese morphism of $X$.   Then $\cF_X=\cL$ is a rank $1$ torsion-free sheaf. We again  take a sufficiently large and divisible \'etale cover of $A$ to kill all the torsion line bundles $Q_{B,i}$ and consider the corresponding base change of the morphism $a_X: X\rightarrow A_X$. For simplicity of notations, we write $f: X\rightarrow A$ the base change and then 
\begin{eqnarray}\label{decom2}f_*\omega_X=\cO_A\oplus\cQ=\cO_A\oplus \cL\oplus\bigoplus_{p_B\in\cC_{\cQ}}p_B^*\cF_B,
\end{eqnarray}
where $B$ are abelian varieties of dimension $1$ or $2$. 
 
\begin{prop}\label{deg8}
Under the above setting, if there exists $p_B\in\cC_{\cQ}$ with $\dim B=1$, then $\deg \varphi_X=\deg f\leq 8$. Moreover, $\deg \varphi_X=8$ iff $\varphi_X$ is a $(\mathbb{Z}/2\mathbb{Z})^3$-cover.
\end{prop}

\begin{proof}
We denote by $\cF_{B_1}$ the direct summand such that $\dim B_1=1$ and let $X\rightarrow C_{B_1}\xrightarrow{f_1} B_1$ be the Stein factorization of natural morphism $X\rightarrow B_1$. Then $f_{1*}\omega_{C_{B_1}}=\cO_{B_1}\oplus \cF_{B_1}$ and $\deg f_1=k_1=1+\rank \cF_{B_1}>1$.

By Lemma \ref{complementary2}, there exists $p_{B_2}\in\cC_{\cQ}$ such that $p_{B_1}\times p_{B_2}$ is an isogeny.

Let $X\rightarrow S_{B_2}\xrightarrow{f_2} B_2$ be the Stein factorization of the natural morphism $X\rightarrow B_2$. Then \begin{eqnarray}\label{decom3}f_{2*}\omega_{S_{B_2}}=\cO_{B_2}\oplus\cF_{B_2}\bigoplus_{p_i': B_2\rightarrow B_i}p_i'^*\cF_{B_i}, \end{eqnarray}where $B_i$ are the abelian varieties in (\ref{decom2}) which are dominated by $B_2$.

We have a commutative diagram:
\begin{eqnarray*}
\xymatrix{
X\ar[r]\ar[d] & A\ar[d]^{\rho}\\
C_{B_1}\times S_{B_2}\ar[r]^{f_1\times f_2} & B_1\times B_2,}
\end{eqnarray*}
which is a sub-representative of $f$.
Let $X'=(C_{B_1}\times S_{B_2})\times_{(B_1\times B_2)} A$. Then we have the factorization of $f$ as $f: X\xrightarrow{t} X'\xrightarrow{g} A$. Note that $g_*\omega_{X'}$ as a direct summand of $f_*\omega_X$ contains $\rho^*(\cF_{B_1}\boxtimes\cF_{B_2})$ as a direct summand. By the uniqueness of the decomposition theorem, we conclude that $\cL\simeq \rho^*(\cF_{B_1}\boxtimes\cF_{B_2})$. Hence both $\cF_{B_1}$ and $\cF_{B_2} $ are of rank $1$. In particular, $f_1$ is a double cover.

By Corollary \ref{intermediate}, we know that $\chi(X, \omega_X)=\chi(X', \omega_{X'})$. Since both $X$ and $X'$ admit the fibration over $C_{B_1}$, we conclude by Riemann-Roch that, for $c\in C_{B_1}$ general,  $\chi(X_c, \omega_{X_c})=\chi(X'_c, \omega_{X'_c})$. Denote by $A_c$ the fiber of $A$ over $f_1(c)\in B_1$ and let $\rho_c: A_c\rightarrow B_2$ be the natural isogeny.
Then, consider the restricted morphism $a_c: X_c\rightarrow X_c'\xrightarrow{a_c'} A_{c}$, we have $h^0_{a_c}(X_c, \omega_{X_c})=h^0_{a_c'}(X'_c, \omega_{X'_c})=h^0_{\mathrm{Id}_{A_c}}(A_c, \rho_c^*\cF_{B_2})$, where $\rho_c^*\cF_{B_2}$ is a rank $1$ torsion-free sheaf on $A_c$. By the same argument as in Section 2 and Proposition \ref{case1}, we conclude that $\deg(X_c/A_c)\leq 4$ and when the equality holds, $X_c\rightarrow A_c$ is a $(\mathbb{Z}/2\mathbb{Z})^2$-cover. In this case, $X_c'$ is birational to $X_c$ and $S_{B_2}\rightarrow B_2$ is also a $(\mathbb{Z}/2\mathbb{Z})^2$-cover. We then conclude that $f$ is a $(\mathbb{Z}/2\mathbb{Z})^3$-cover.

 \end{proof}

Let's recall that a threefold of maximal Albanese dimension, of general type, with $\chi(X, \omega_X)=0$ is called an Ein-Lazarsfeld threefold and its Albanese morphism is a $(\mathbb{Z}/2\mathbb{Z})^2$-cover (see \cite{CDJ}).

\begin{prop}\label{EL}Assume that $g: Y\rightarrow A$ be the Albanese morphism of an Ein-Lazarsfeld threefold. Assume that for some generically finite morphism $t: X\rightarrow Y$ of degree $\geq 2$, the induced morphism $f=g\circ t: X\rightarrow A$ is birationally equivalent to the eventual paracanonical map of $X$, then $f$ is a birational $(\mathbb{Z}/2\mathbb{Z})^3$-cover.
\end{prop}
\begin{proof}
Assume that $f$ is the eventual paracanonical map of $X$,  then  $f$ is the Albanese morphism of $X$. We are free to take abelian \'etale covers of $A_X$ and consider the base change. Hence we will assume that no torsion line bundles appear in the decomposition formula.

 If $t$ is not minimal, then it factors birationally as $X\rightarrow Z\rightarrow  Y$ with $\deg(X/Z)\geq 2$ and $\deg(Z/Y)\geq 2$, then $\chi(Z, \omega_Z)>0$ by the structure of Ein-Lazarsfeld variety and hence by Corollary \ref{intermediate}, $\chi(Z, \omega_Z)=\chi(X, \omega_X)>0$.  However, $\deg(X/A)> 8$, which is a contradiction to Theorem \ref{=}. Hence $t$ is minimal.

We write $t_*\omega_X=\omega_Y\oplus \cQ$. Then $g_*\cQ$ has rank $\geq 4$, while $\cF_X=\cL$ is of rank $1$. Hence there exists $p_B\in \cC_{\cQ}$. Consider the induced morphism $X_B\rightarrow Y_B$, which is not birational. By the minimality of the extension $\mathbb{C}(X)/\mathbb{C}(Y)$, we conclude that $X$ is birational to the main component of $X_B\times_{Y_B}Y$. Denote by $a=\deg(X_B/Y_B)=\deg (X/Y)$.

For all $p_B\in\cC_{\cQ}$, either $\deg(Y_B/B)=2$ or $Y_B$ is birational to $B$. Hence either $\deg(X_{B}/B)=2a$ or $\deg(X_B/B)=a$ and $\rank\cF^{\cQ}_B=2a-2$ or $a-1$.

 Since  $$\rank f_*\omega_X=\rank  g_*\omega_Y+\rank \cL+\sum_{p_B\in\cC_{\cQ}}\rank \cF_{B}^{\cQ},$$ we conclude that $a-1$ divides $\rank \cL=1$. Thus  $a=2$ and $X$ is birational to the main component of $X_B\times_{Y_B}Y$ for any $p_B\in\cC_{\cQ}$. Hence $f: X\rightarrow A$ is a birational $(\mathbb{Z}/2\mathbb{Z})^3$-cover.
 \end{proof}

\begin{prop}\label{deg4}
Assume that the Albanese morphism $a_X: X\rightarrow A_X$ is birationally equivalent to $\varphi_X$,  $\dim B=2$ for all $p_B\in\cC_{\cQ}$, and $a_X$ does not factors though any variety of general type with vanishing holomorphic Euler characteristic, then $a_X$ is a $(\mathbb{Z}/2\mathbb{Z})^2$-cover.
\end{prop}

\begin{proof} 
First we take $Z$ a threefold of general type over $A_X$ with $\deg(Z/A_X)$ minimal such that $a_X$ factors as $X\rightarrow Z\rightarrow A_X$. By assumption, $\chi(Z, \omega_Z)>0$. Hence $\chi(X, \omega_X)=\chi(Z, \omega_Z)$. By Theorem \ref{=} and the assumption that $X$ does not factors through Ein-Lazarsfeld varieties, we conclude that either $X$ birational to $Z$ or $a_X$ is a birational bidouble cover. Hence in the following we shall assume that $X$ birational to $Z$, namely $a_X$ does not factors non-trivially through any other varieties of general type.

Then for any $p_{B_i}, p_{B_j}\in\cC_{\cQ}$, let $\PE_{ij}$ be the neutral component of $\PB_i\cap\PB_j$. Then a desingularization $X_{ij}$ of the main component of the fiber product of $X_{B_i}\times_{E_{ij}}X_{B_j}$ is of general type.  We denote by $X_{ij}'$ the abelian \'etale cover of $X_{ij}$ such that $a_X$ factors as $X\xrightarrow{t_{ij}} X_{ij}'\rightarrow A$. Then $t_{ij}$ is birational.

 Let $r_i=\rank \cF_{B_i}$. Then $\deg(X_{B_i}/B_i)=1+r_i$ and $(1+r_i)(1+r_j)=\deg a_X$.

If $\cC_{\cQ}=\{p_{B_1}, p_{B_2}\}$, then $(1+r_1)(1+r_2)=2+r_1+r_2$. The only solution is $r_1=r_2=1$ and hence $f$ is a $(\mathbb{Z}/2\mathbb{Z})^2$-cover.  If $|\cC_{\cQ}|=s\geq 3$, then $r_i=r_1$ for all $i$ and hence  $(1+r_1)^2=2+r_1s$. Then $r_1^2+2r_1-1=r_1s$. Hence $r_1=1$ and $s=2$, which is a contradiction.  
 
\end{proof}

By Proposition \ref{deg8}, \ref{EL}, \ref{deg4}, we have
\begin{theo}\label{k=0}Let $X$ be a threefold of maximal Albanese dimension with $\chi(X, \omega_X)>0$. Assume that $a_X: X\rightarrow A_X$ is birationally equivalent to $\varphi_X$, then $\deg a_X\leq 8$ and the equality holds only when $a_X$ is a birational $(\mathbb{Z}/2\mathbb{Z})^3$-cover.
\end{theo}

\end{document}